\documentclass{amsart}

\usepackage{amssymb} \usepackage{amsfonts} \usepackage{amsmath}
\usepackage{amsthm} \usepackage{epsfig} 
\usepackage{color}
\usepackage{amscd}

\newtheorem{lemma}{Lemma}[section]
\newtheorem{teo}[lemma]{Theorem}
\newtheorem{prop}[lemma]{Proposition}
\newtheorem{cor}[lemma]{Corollary}

\theoremstyle{definition}
\newtheorem{defn}[lemma]{Definition}

\newtheorem{quest}[lemma]{Question}

\theoremstyle{remark}
\newtheorem{rem}[lemma]{Remark}

\newcommand{\matN}{\ensuremath {\mathbb{N}}}
\newcommand{\matR} {\ensuremath {\mathbb{R}}}

\newcommand{\matZ} {\ensuremath {\mathbb{Z}}}
\newcommand{\matC} {\ensuremath {\mathbb{C}}}

\newcommand{\matD} {\ensuremath {\mathbb{D}}}

\newcommand{\calL} {\ensuremath {\mathcal{L}}}

\newcommand{\calM} {\ensuremath {\mathcal{M}}}

\newcommand{\calH}{\ensuremath {\mathcal{H}}}
\newcommand{\calO}{\ensuremath {\mathcal{O}}}

\newcommand{\isom}{\cong}
\newcommand{\Hom}{{\rm Hom}}
\newcommand{\End}{{\rm End}}
\newcommand{\St}{{\rm St}}

\newcommand{\cerchio}{\includegraphics[width = .4 cm]{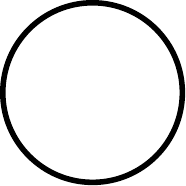}}
\newcommand{\teta}{\includegraphics[width = .4 cm]{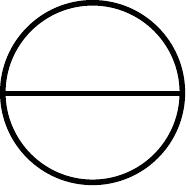}}
\newcommand{\tetra}{\includegraphics[width = .4 cm]{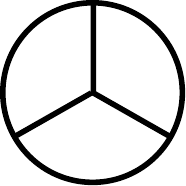}}
\newcommand{\tetracolored}{\includegraphics[width = 1 cm] {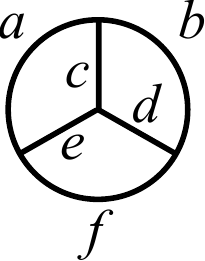}}

\newcommand{\MCG}{{\rm Mod}}
\newcommand{\id}{{\rm id}}
\newcommand{\SL}{{\rm SL}}
\newcommand{\GL}{{\rm GL}}
\newcommand{\SU}{{\rm SU}}
\newcommand{\YM}{\mathcal{Y}\!\mathcal{M}}

\newcommand{\Comm}{{\rm Comm}}

\author{Francesco Costantino}
\address{IRMA, 7 rue Ren\'e Descartes 67087, Strasbourg, France}
\email{costanti at math dot unistra dot fr}

\author{Bruno Martelli}
\address{Dipartimento di Matematica ``Tonelli'', Largo Pontecorvo 5, 56127 Pisa, Italy}
\email{martelli at dm dot unipi dot it}

\thanks{The first author was supported by the French ANR project ``Quantum G\&T'', ANR-08-JCJC-0114-01. The second author was supported by the Italian FIRB project ``Geometry and topology of low-dimensional manifolds'', RBFR10GHHH}

\title[Analytic family of representations]{An analytic family of representations for the mapping class group of punctured surfaces}

\begin{document}

\begin{abstract}
We use quantum invariants to define an analytic family of representations for the mapping class group $\MCG(\Sigma)$ of a punctured surface $\Sigma$. The representations depend on a complex number $A$ with $|A|\leqslant 1$ and act on an infinite-dimensional Hilbert space. They are unitary when $A$ is real or imaginary, bounded when $|A|<1$, and only densely defined when $|A| =1$ and $A$ is not a root of unity. When $A$ is a root of unity distinct from $\pm 1$ and $\pm i$ the representations are finite-dimensional and isomorphic to the ``Hom'' version of the well-known TQFT quantum representations.

The unitary representations in the interval $[-1,0]$ interpolate analytically between two natural geometric unitary representations, the \emph{$\SU(2)$-character variety} representation studied by Goldman and the \emph{multicurve} representation induced by the action of $\MCG(\Sigma)$ on multicurves.

The finite-dimensional representations converge analytically to the infinite-dimensional ones.
We recover March\'e and Narimannejad's convergence theorem, 
and Andersen, Freedman, Walker and Wang's asymptotic faithfulness, that states that 
the image of a non-central mapping class is always non-trivial after some level $r_0$. When the mapping class is pseudo-Anosov we give a simple polynomial estimate of the level $r_0$ in term of its dilatation.
\end{abstract}

\maketitle

\section{Introduction}

Let $\Sigma$ be a closed, oriented, connected surface with some marked points (or equivalently, some punctures) and let $\MCG(\Sigma)$ be the mapping class group of $\Sigma$. We assume that $\Sigma$ has at least one marked point and that by removing them from $\Sigma$ we get a surface of negative Euler characteristic. Recall that a \emph{multicurve} in $\Sigma$ is a finite collection of disjoint non-trivial simple closed curves, considered up to isotopy.\footnote{A curve is \emph{trivial} when it bounds a disc in $\Sigma$ that does not contain any marked point, hence puncture-parallel curves are admitted. We also consider the empty set as a multicurve.} The group $\MCG(\Sigma)$ acts on the set $\calM$ of all multicurves and hence induces a unitary representation 
$$\rho_0\colon \MCG(\Sigma) \to U(\calH)$$
on the Hilbert space $\calH=\ell^2(\calM)$ which we call the \emph{multicurve representation}. Let $\matD\subset \matC$ denote the open unit disc and $\overline \matD$ its closure. We construct here a family of representations $\{\rho_A\}_{A\in \overline \matD}$ that include $\rho_0$ and vary analytically in $A$ in an appropriate sense. 

\subsection{The analytic family}

Let $x$ be a fixed triangulation of $\Sigma$ having its vertices at the marked points.
For every integer $r \geqslant 2$ let $\calM_r\subset \calM$ be the set of all multicurves on $\Sigma$ having a representative that intersects every triangle of $x$ in at most $r-2$ arcs, and let $\calH_r\subset \calH = \ell^2(\calM)$ consist of all $\mathbb{C}$-valued functions supported on $\calM_r$. It is easy to check that $\calM_r$ is finite and hence $\calH_r$ is finite-dimensional; moreover if $r<r'$ then $\calH_r \subsetneq \calH_{r'}$ and $\check\calH = \cup_r \calH_r$ is the space of all finitely-supported functions on $\calM$, a dense subset in $\calH$. 

We denote by $B(\calH)$ and $U(\calH)$ respectively the space of all bounded invertible and unitary linear operators $\calH \to \calH$. If $A$ is a root of unity we set $r=r(A)$ to be the smallest integer such that $A^{4r}=1$.

We define here
for every complex number $A \in \overline \matD$ a linear representation $\rho_A$ of $\MCG(\Sigma)$ on a subspace of $\calH$ which depends on $A$ as follows:
\begin{itemize}
\item if $A\in \matD$ we define a bounded representation
$$\rho_A\colon \MCG(\Sigma) \to B(\calH),$$
\item if $A\in \overline \matD \cap (\matR \cup i \matR)$ we get a unitary representation
$$\rho_A\colon \MCG(\Sigma) \to U(\calH),$$
\item if $|A|=1$ and $A$ is not a root of unity we get a densely defined representation
$$\rho_A\colon \MCG(\Sigma) \to \GL(\check\calH),$$
\item if $A$ is a root of unity distinct from $\pm 1$ and $\pm i$ we define a finite-dimensional representation
$$\rho_A\colon \MCG(\Sigma) \to \GL(\calH_{r(A)})$$
which is also unitary when $A= \pm \exp(\frac{i\pi}{2r})$.
\end{itemize}

When $|A|=1$ and $A$ is not a root of unity we do not know if $\rho_A$ is bounded or unbounded: if it were bounded it would of course extend continuously to $\calH$. The various types of representations are summarized in Fig.~\ref{representationsB:fig}.

\begin{figure}
 \begin{center}
  \includegraphics[width = 12 cm]{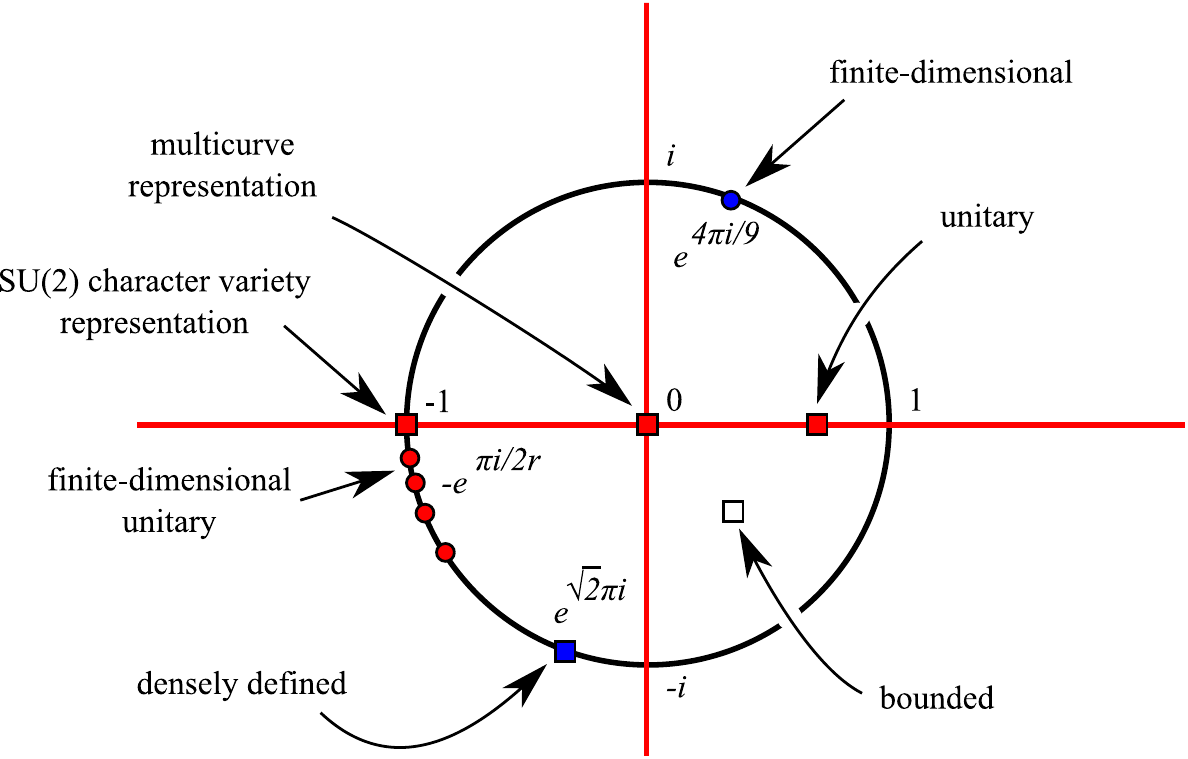}
 \end{center}
 \caption{The analytic family of representations.}
 \label{representationsB:fig}
\end{figure}

Throughout this paper we will denote by $S\subset \partial \matD$ the set of all roots of unity except $\pm 1$ and $\pm i$: the representation $\rho_A$ is finite-dimensional precisely when $A \in S$.
Although its domain changes dramatically with $A$, the representation $\rho_A$ varies analytically in $A\in\overline \matD$ in an appropriate sense: 

\begin{teo} \label{main2:teo}
For any pair of vectors $v,v'\in \calH$ and any element $g\in\MCG(\Sigma)$ the matrix coefficient 
$$\langle\rho_A(g)v,v'\rangle$$ 
depends analytically on $A\in \matD$.
Moreover, if $v,v'\in \check\calH$ there is a finite set $F \subset S$ such that 
$\langle\rho_A(g)v,v'\rangle$
is defined and varies analytically on $\overline \matD \setminus F$.
\end{teo}
Therefore the family $\{\rho_A\}_{A\in\matD}$ of bounded representations is analytic in $\calH$ in the usual sense \cite{PS}. The whole family $\{\rho_A\}_{A\in \overline \matD}$ includes densely defined and finite representations and is ``analytic'' in the strongest possible sense: note that we necessarily need to exclude a finite set $F$ of points $A\in S$ since the matrix coefficient makes sense only if $v$ is contained in $\calH_{r(A)}$ and this holds only for sufficiently big values of $r(A)$. 
The finite set $F$ of course depends on $v, v',$ and $g$. 

Theorem \ref{main2:teo} says in particular that the finite-dimensional representations converge analytically to the infinite-dimensional ones:
\begin{cor}
Let $A_i\in S$ be a sequence that converges to an element $A_\infty\in \partial\matD\setminus S$. 
For any pair $v,v'\in \check\calH$ of vectors and any element $g\in\MCG(\Sigma)$ there is a number $N>0$ such that the matrix coefficient 
$$\langle\rho_{A_i}(g)v,v'\rangle$$
is defined and varies analytically on $\{A_i\}_{i\geqslant N} \cup \{A_\infty\}$. In particular, the matrix coefficients converge 
$$\langle\rho_{A_i}(g)v,v'\rangle \longrightarrow \langle\rho_{A_\infty}(g)v,v'\rangle$$
as $A_i \to A_\infty$.
\end{cor}

When $A_j= \pm \exp (\frac{i\pi}{2j}) \to \pm 1$ the finite-dimensional unitary representations $\rho_{A_j}$ converge to the unitary infinite-dimensional representation $\rho_{\pm 1}$ in an analytic sense, which implies the usual Fell convergence\footnote{Here we do not assume the representations to be irreducible as in the standard Fell topology.}. We recover here (for punctured surfaces) an important convergence result of March\'e and Narimannejad \cite{MN}.

The family of representations is constructed combining an important ingredient of quantum topology called \emph{quantum $6j$-symbol} with Valette's \emph{cocycle technique} \cite{V}. The cocycle technique was introduced to produce an analytic family of representations for groups acting on trees \cite{V} and was later applied successfully to some groups acting on more complicate complexes, see the discussion in Section \ref{introduction:cocycle:subsection}. Here we use the action of $\MCG(\Sigma)$ on the \emph{triangulation graph} of $\Sigma$ and construct a cocycle using the (suitably renormalized) quantum $6j$-symbols. We use and slightly improve a crucial estimate of Frohman and Kania-Bartoszy\'nska \cite{FK2} to show that the cocycle indeed perturbs $\rho_0$ and yields bounded representations when $A\in\matD\setminus (\matR\cup i \matR)$.

The cocycle technique typically produces an analytic family of representations on a fixed Hilbert space $\calH$. At the best of our knowledge, the family $\{\rho_A\}_{A\in \overline\matD}$ constructed here is the first analytic family of representations of a discrete group which include finite-dimensional and infinite-dimensional representations.

As usual with cocycles, the family $\rho_A$ mildly depends on the initial choice of a fixed element of the complex, that is the triangulation $x$ for $\Sigma$: see Remark \ref{dependenceonx:rem}. 

\subsection{Relations with the Kauffman algebra}
Recall that two representations $\rho\colon G \to \GL(V)$ and $\rho'\colon G \to \GL(V')$ are \emph{isomorphic} (\emph{isometric}) if they are connected by a linear isomorphism (isometry) $V \to V'$. 

When $A\in \matC^*$ the group $\MCG(\Sigma)$ acts naturally on a rich algebraic object $K_A(\Sigma)$ called the \emph{skein algebra} of $\Sigma$. This object is a $\matC$-algebra equipped with a trace (known as Yang-Mills trace) and hence a complex bilinear form \cite{YM}, see Section \ref{relationskein:sec}. We show here the following.
\begin{prop} \label{skein:prop}
If $A\in \matR^*$ the Yang-Mills trace induces a positive definite hermitian form and hence a pre-Hilbert space structure on $K_A(\Sigma)$. If $A\in [-1,0)\cup (0,1]$ the representation $\rho_A$ is isometric to the natural action of $\MCG(\Sigma)$ on the Hilbert space $\overline{K_A(\Sigma)}$ obtained by completing $K_A(\Sigma)$.
\end{prop}
As a consequence, the unitary representation $\rho_{-1}$ is isometric to a well-known geometric  representation, which we now describe. The group $\MCG(\Sigma)$ acts naturally on the $\SU(2)$-character variety 
$$X=\Hom\big(\pi_1(\Sigma),\SU(2)\big)/_{\SU(2)}$$
which is equipped with a natural Haar measure $\mu$; hence $\MCG(\Sigma)$ acts on the Hilbert space $L^2(X, \mu)$. We call this action the \emph{$\SU(2)$-character variety representation}. The Hilbert space $L^2(X, \mu)$ contains as a dense subspace the algebra $T$ of all \emph{trace functions} on $X$, and Bullock and Charles-March\'e \cite{Bu, ChMa} have constructed a $\MCG(\Sigma)$-equivariant algebra isomorphism between $K_{-1}(\Sigma)$ and $T$. In Section \ref{subsec:unitaryreps} we provide a proof of the following well-known:

\begin{cor} \label{character:cor}
The representation $\rho_{-1}$ is isometric to the $\SU(2)$-character variety representation.
\end{cor}
As a consequence we see that the representation $\rho_A$ is unitary for $A\in\matR$ and so the segment $[-1,0]$ defines a path connecting the $\SU(2)$-character variety representation to the multicurve representation:
\begin{cor} 
There is an analytic path of unitary representations connecting the $\SU(2)$-character variety representation to the multicurve representation.
\end{cor}

We also mention that the finite-dimensional representations $\{\rho_A\}_{A\in S}$ are isomorphic to the ``Hom'' version of the well-known \emph{finite} representations arising from $\SU(2)$ topological quantum field theory \cite{BHMV}, see Section \ref{finite:subsection} for a precise statement.

Following the same line it is possible to relate $\rho_{\sqrt {-1}}$ to a representation studied by March\'e in \cite{Ma}. Note that $\{\rho_{t\sqrt{-1}}\}_{t \in [0,1]}$ interpolate analytically between $\rho_{\sqrt {-1}}$ and $\rho_0$.

\subsection{Faithfulness}
Recall that the center of $\MCG(\Sigma)$ is always trivial, except when $\Sigma$ is a punctured torus, and in that case it is a cyclic group of order 2 generated by the elliptic involution.

\begin{teo}\label{faithfullness:teo}
For every non-central element $g\in \MCG(\Sigma)$ there is a finite set $F\subset S$ such that $\rho_A(g)\neq \id$ for every $A\in\overline \matD\setminus F$.
\end{teo}

In particular every infinite-dimensional representation $\rho_A$ is faithful (modulo the center), and for finite-dimensional representations we recover the well-known \emph{asymptotic faithfulness} proved by Andersen \cite{A} and Freedman, Walker and Wang \cite{FWW}, which may be restated as follows: for every non-central element $g\in\MCG(\Sigma)$ there is a level $r_0$ such that $\rho_A(g)\neq \id$ for all $A$ with $r(A)\geqslant r_0$. When $g$ is pseudo-Anosov we may estimate the level $r_0$ polynomially on the dilatation $\lambda$:

\begin{teo} \label{main:dilatation:teo}
Let $\Sigma$ be a punctured surface and $g\colon \Sigma \to \Sigma$ be a pseudo-Anosov map with dilatation $\lambda > 1$. If $A$ is a primitive $(4r)^{\rm th}$ root of unity with
$$r>-6\chi(\Sigma)\big(\lambda^{-9\chi(\Sigma)}-9\chi(\Sigma)-1\big) +1$$
then
$$\rho_A(g)\neq \id.$$
\end{teo}
The proof of Theorem \ref{main:dilatation:teo} (and of all the preceding theorems) makes an essential use of punctures: we do not know if a similar lower bound holds for closed surfaces.

\subsection{Irreducible components}
Recall that $\calH = \ell^2(\calM)$ where $\calM$ is the set of all multicurves seen up to isotopy. Split $\calM$ as $\calM = \calM^0 \cup \calM^{\neq 0}$ where $\calM^0$ (resp.~$\calM^{\neq0}$) consists of all multicurves that are trivial (resp.~non-trivial) in $H_1(\Sigma, \matZ_2)$.

The Hilbert space $\calH$ splits accordingly into two infinite-dimensional orthogonal subspaces:
$$\calH = \calH^0\oplus \calH^{\neq 0}$$
where $\calH^0$ (resp.~$\calH^{\neq0}$) consists of all functions that are supported on $\calM^0$ (resp.~$\calM^{\neq 0}$).
Each finite-dimensional subspace $\calH_r$ decomposes similarly as
$$\calH_r = \calH^{0}_r\oplus \calH^{\neq 0}_r$$
with $\calH_r^0 = \calH^0 \cap \calH_r$ and $\calH_r^{\neq 0} = \calH^{\neq 0} \cap \calH_r$.
\begin{prop}\label{orthogonalsubspaces:prop}
The closed subspaces $\calH^0$ and $\calH^{\neq 0}$ are $\rho_A$-invariant for every $A \in \overline \matD\setminus S$. Similarly $\calH_r^0$ and $\calH^{\neq 0}_r$ are $\rho_A$-invariant for any $A\in S$ with $r(A)=r$.
\end{prop}
Therefore every representation splits into two factors. Such a splitting was already noticed by Goldman in \cite{GoRed} for the $\SU(2)$-character variety representation, \emph{i.e.}~for $\rho_{-1}$. Goldman also asked whether it is possible to further decompose each factor into irreducibles. We extend this question to every real value $A$:

\begin{quest}
Do the representations $\rho_A$ restricted to $\calH^0$ and $\calH^{\neq 0}$ split into irreducible components when $A\in [-1,1]$? What is this decomposition?
\end{quest}

We can answer this question only when $A=0$, \emph{i.e.}~for the multicurve representation $\rho_0$ on $\ell^2(\calM)$. The action of $\MCG(\Sigma)$ on $\calM$ decomposes into infinitely many orbits $\{\calM_O\}_{O\in \calO}$, and $\calH$ decomposes accordingly as
$$\calH = \bigoplus_{O\in \calO} \calH_O$$
where $\calH_O$ consists of all functions supported on $\calM_O$. Each $\calH_O$ is obviously $\MCG(\Sigma)$-invariant. The action of $\MCG(\Sigma)$ on a factor $\calH_O$ needs not to be irreducible, but a result of Paris \cite{Pa} based on Burger-De La Harpe \cite{BD} easily implies the following:

\begin{teo} \label{irreducibles:teo}
Each factor $\calH_{O}$ splits into finitely many orthogonal irreducibles. 
\end{teo}

\subsection{Structure of the paper}
In Section \ref{repres:section} we construct the family $\{\rho_A\}_{A\in \overline\matD\setminus S}$ of infinite-dimensional representations using the cocycle technique, then we prove the first statement of Theorem \ref{main2:teo}: to do so we assume some properties and estimates of $6j$-symbols summarized in Proposition \ref{exist:renormalized:prop} and proved in Section \ref{estimate:section}. 

Section \ref{Kauffman:section} introduces the skein module of a 3-manifold and quantum $6j$-symbols following \cite{L_book}: the expert reader may skip it or just refer to it for notations. In Section \ref{estimate:section} we introduce the \emph{renormalized $6j$-symbol} and prove Proposition \ref{exist:renormalized:prop}. 

The finite-dimensional representations are then introduced in Section \ref{finite:section} and in Section \ref{sub:teo1.5} we prove the second statement of Theorem \ref{main2:teo} (see Corollary \ref{cor:main2teo}) and half of Theorem \ref{faithfullness:teo}. The representations are reinterpreted in terms of the Kauffman skein algebra in Section \ref{relationskein:sec}, where in Section \ref{revisited:subsection} we prove the other half of Theorem \ref{faithfullness:teo} and at the end of Section \ref{subsec:unitaryreps} we prove Proposition \ref{skein:prop}. 

In Section \ref{sec:anosov} we deal with pseudo-Anosov mapping classes and prove Theorem \ref{main:dilatation:teo}. In Section \ref{sec:irreducibles} we prove Proposition \ref{orthogonalsubspaces:prop} and Theorem \ref{irreducibles:teo}.

\subsection{Acknowledgements} We thank Julien March\'e, Gregor Masbaum, and Alain Valette for useful conversations. The first author was supported by the French ANR Project ANR-08-JCJC-0114-01.

\section{The infinite-dimensional representations}\label{repres:section}
We construct here the analytic family $\{\rho_A\}_{A\in \overline \matD \setminus S}$ of infinite-dimensional representations by combining Pimsner and Valette's cocycle technique \cite{V, Pim} with quantum 6j-symbols \cite{KL}. We prove the first statement of Theorem \ref{main2:teo} (see Theorem \ref{proofthm1:sub}) using the properties of quantum $6j$-symbols that we summarize in Proposition \ref{exist:renormalized:prop} and prove in Section \ref{estimate:section}.

The finite-dimensional representations are constructed in Section \ref{finite:section}. The relations between these representations and the well-known skein module representations are later depicted in Section \ref{relationskein:sec}. 

\subsection{Analytic families of representations}
\label{introduction:cocycle:subsection}

The literature contains some interesting examples of analytic families of representations, the most influential one being probably the family $\{\rho_z\}_{z \in \matD}$ constructed by Pytlik and Szwarc in 1986 for free non-abelian groups \cite{PS}. In their example $\rho_0$ is the regular representation of $G$ in $\ell^2(G)$, and the family $\{\rho_z\}_{z\in \matD}$ may thus be interpreted as an analytic perturbation of $\rho_0$; the representation $\rho_z$ is unitary for real values of $z$ and uniformly bounded elsewhere. 

Valette \cite{V} then noticed that when $G$ acts on some space $X$ one may try to perturb a given representation by constructing a suitable \emph{cocycle} 
$$c\colon X\times X \to B(\calH),$$
first considered in \cite{Pim}.
Valette reinterpreted Pytlik and Szwarc's construction using cocycles and extended it to any group $G$ acting on a tree. 
This cocycle technique (described below) has later proved successful on some other groups acting on more complicate cell complexes, including: 
\begin{itemize}
\item right-angled Coxeter groups, by Januszkiewicz \cite{J},
\item groups acting on a CAT(0) cubical complex, by Guenter and Higson \cite{GH}. 
\end{itemize}

We show here that the cocycle technique also applies to the mapping class group $G=\MCG(\Sigma)$ of a punctured surface $\Sigma$. The base representation $\rho_0$ is the multicurve representation on $\calH = \ell^2(\calM)$ and we perturb it by considering the action of $G$ on the \emph{triangulations graph} $X$ of $\Sigma$.

\subsection{The cocycle technique} \label{technique:subsection}
Recall that $B(\calH)$ is the space of all bounded linear operators and $U(\calH)$ is the multiplicative group of all isometric linear operators on a complex Hilbert space $\calH$. Let $G$ be a group acting on a space $X$ and $\rho\colon G \to U(\calH)$ be a unitary representation. 
\begin{defn} \label{cocycle:defn}
A \emph{cocycle} is a map $c\colon X \times X \to B(\calH)$ such that:
\begin{enumerate}
\item $c(x,x) = \id$,
\item $c(x,y)\cdot c(y,z) = c(x,z)$,
\item $c(gx, gy) = \rho(g) c(x,y) \rho(g)^{-1}$
\end{enumerate}
for all $x,y \in X$ and $g\in G$. 
\end{defn}

If $c$ is a cocycle and $x \in X$ is a point, then the formula
$$\rho_c(g) = c(x,gx) \rho(g)$$
defines a representation $\rho_c\colon G\to B(\calH)$. If the cocycle is \emph{unitary}, meaning that $c(x,y) = c(y,x)^*$ for all $x$ and $y$, then the representation $\rho_c$ is unitary. 

A family $\{c_z\}_{z\in \matD}$ of cocycles is \emph{analytic} if the function
$$z\mapsto \langle c_z(x,y)v,w\rangle $$ 
is analytic in $z$ for any $x,y \in X$ and any vectors $v,w\in\calH $.
An analytic family $c_z$ of cocyles gives rise to an analytic family $\rho_z$ of representations. If $c_0$ is the trivial cocycle, \emph{i.e.}~$c_0(x,y) = \id$ for all $x,y$, then $\rho_0 = \rho$ and we have obtained an analytic perturbation of $\rho$.

A basic example is the construction of Valette \cite{V} for groups acting on trees, which is worth recalling here. Let $G$ be a group that acts on a tree $T$, and $X=T^0$ be the set of its vertices; the Hilbert space is $\calH = \ell^2(X)$ and the action $\rho$ on $\calH$ is the permutation representation. The space $\calH$ has an obvious Hilbert basis $\{\delta_x\}_{x\in X}$.

Consider the holomorphic function $w=f(z) = \sqrt {1-z^2}$ on the open unit disc $\matD$ defined unambiguously by setting $f(0)=1$. For any pair $x,y\in X$ of adjacent vertices (\emph{i.e.}~two endpoints of an edge) define the cocycle $c_z(x,y)\in B(\ell^2(X))$ as the endomorphism which acts as the matrix
$$\begin{pmatrix} w & z \\ -z & w \end{pmatrix}$$
on the plane spanned by the basis $(\delta_x, \delta_y)$ and acts as the identity on its orthogonal complement. Note that $c_z(x,y) = c_z(y,x)^{-1}$, and that $c_z(x,y) = c_z(y,x)^*$ precisely when $z$ is real. We then define $c_z(x,y)$ on an arbitrary pair of edges by setting
$$c_z(x,y) = c_z(x_1,x_2) \cdots c_z(x_{n-1},x_n)$$
for any path of adjacent vertices $x=x_1, x_2, \ldots, x_{n-1}, x_n=y$ connecting $x$ and $y$. We have defined an analytic family of cocycles and hence an analytic family $\rho_z$ of representations for $G$, which are unitary when $z$ is real. The cocycle $c_0$ is trivial and hence the representation $\rho_0$ is the permutation representation on $X=T^0$.

\subsection{The triangulations graph}
We apply the cocycle technique to our setting. Here $\Sigma$ is a surface with marked points and $G=\MCG(\Sigma)$ is its mapping class group (\emph{i.e.}~the group of all orientation preserving diffeomorphisms of $\Sigma$ fixing the set of marked points, seen up to isotopies also fixing that set). We define the space $X$ as the \emph{triangulations graph} of $\Sigma$, \emph{i.e.}~the space of all triangulations of $\Sigma$ having the marked points as vertices, seen up to isotopy keeping the vertices fixed. The set $X$ is the vertex set of a natural connected graph, constructed by joining with an edge any two triangulations that differ only by a \emph{flip} as in Fig.~\ref{flip:fig}. The mapping class group $\MCG(\Sigma)$ clearly acts on the graph $X$ with compact quotient. However, the graph of triangulations is not a tree, and hence the construction described above does not apply \emph{as is}.

\begin{figure}
 \begin{center}
  \includegraphics[width = 6 cm]{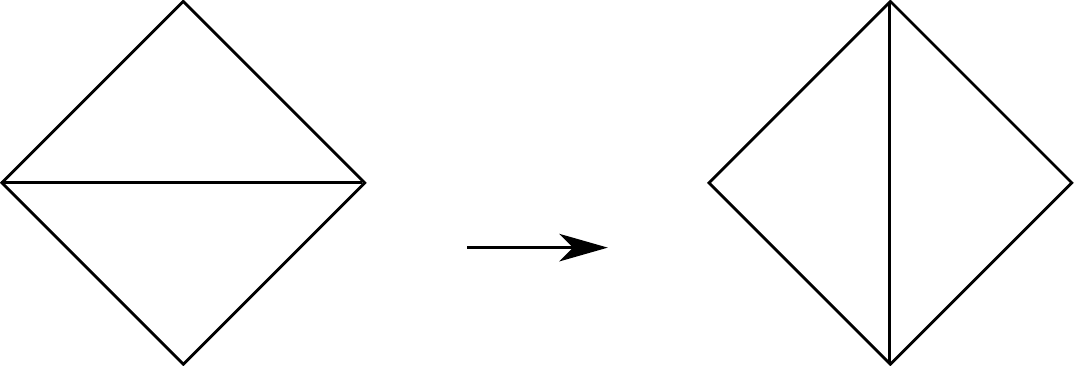}
 \end{center}
 \caption{A flip.}
 \label{flip:fig}
\end{figure}

Constructing a non-trivial cocycle in this setting is a non-trivial problem. Let $x$ and $y$ be two triangulations and $x=x_1, x_2, \ldots, x_{n-1}, x_n=y$ be a path in the graph $X$ joining $x$ and $y$. The following equality must hold
$$c(x,y) = c(x_1,x_2) \cdots c(x_{n-1},x_n)$$
which decomposes $c(x,y)$ into factors $c(x_i,x_{i+1})$, each corresponding to a flip. The issue here is that the path joining $x$ to $y$ is far from being unique. There are three well-known moves that relate any two different paths, see \cite{Pen0} for a proof:

\begin{teo} \label{pentagon:teo}
Two different paths joining two fixed triangulations $x$ and $y$ are connected by the following types of  moves:
\begin{enumerate}
\item
addition or removal of a pair of consecutive flips in the same quadrilateral,
\item
two commuting flips on disjoint quadrilaterals,
\item
the \emph{pentagon relation} in Fig~\ref{pentagon:fig}.
\end{enumerate}
\end{teo}

\begin{figure}
 \begin{center}
  \includegraphics[width = 10cm]{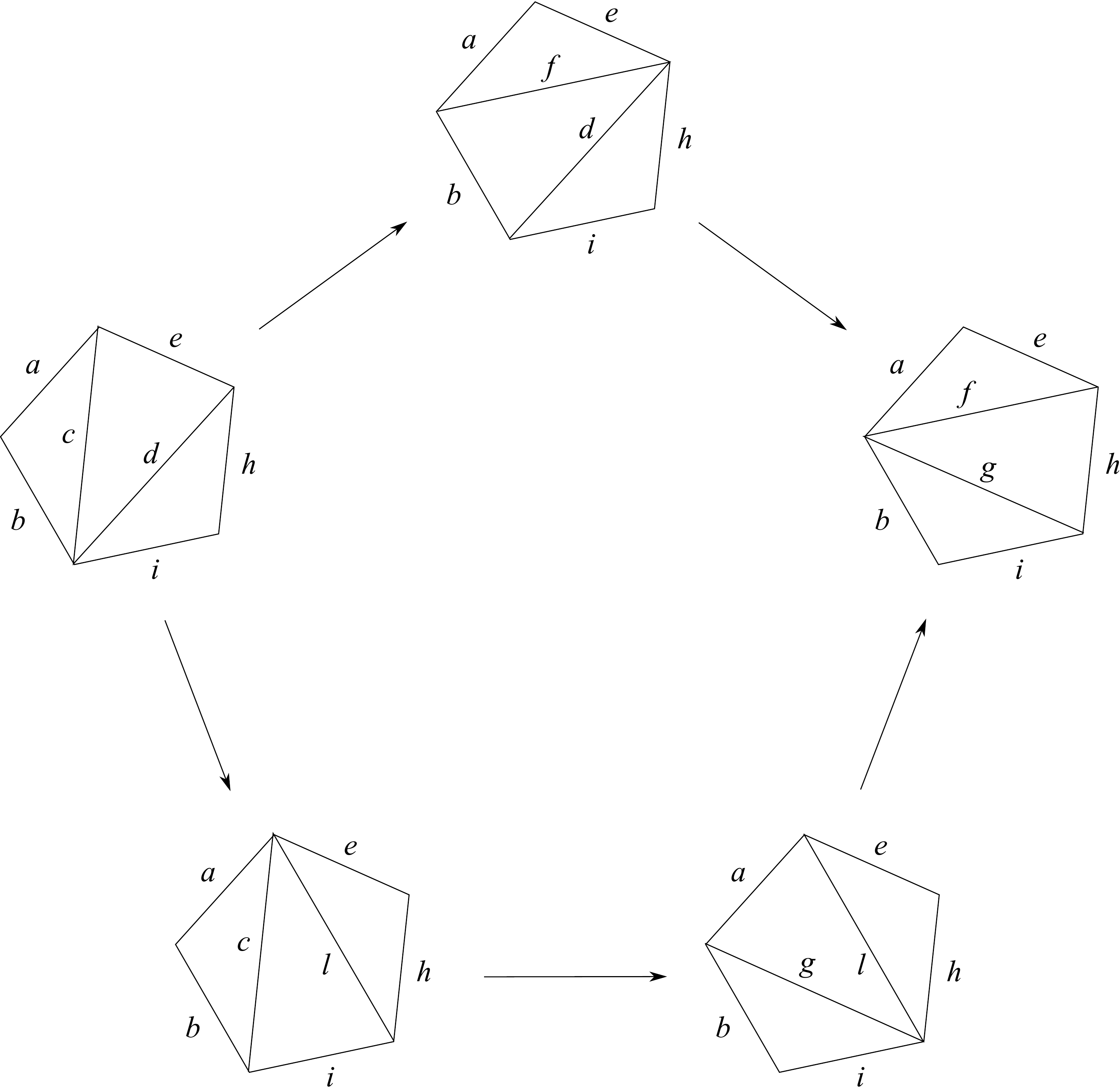}
 \end{center}
 \caption{The pentagon relation.}
 \label{pentagon:fig}
\end{figure}

A cocycle on the graph $X$ of triangulations must be invariant under these three moves, the most important one being the pentagon relation. Quantum invariants indeed provide an analytic family of objects that are invariant under moves (1)-(3), the \emph{quantum 6j-symbols} \cite{KL}. 

\subsection{Quantum $6j$-symbols}
From now on we will use the symbol $A$ instead of $z$ to denote a complex variable typically belonging to the closed unit disc $\overline \matD$, as customary in quantum topology. Recall that $S$ is the set of all roots of unity except $\pm 1$ and $\pm i$.
The \emph{quantum $6j$-symbol} is a rational function in $A$ with poles contained in $S\cup \{0,+\infty\}$ denoted by
$$\begin{Bmatrix}  a & b & c \\ d & e & f \end{Bmatrix}$$
which depends on six non-negative integers $a,b,c,d,e,f$ (an explicit formula is provided before Definition \ref{6j:defn}). The $6j$-symbols satisfy the \emph{orthogonality relations}:
$$\sum_{f} \begin{Bmatrix}  a & b & c \\ d & e & f \end{Bmatrix}
\begin{Bmatrix}  a & b & g \\ d & e & f \end{Bmatrix} = \delta_{cg}
$$
and the \emph{Biedenharn-Elliot identity:}
$$\begin{Bmatrix}  a & b & c \\ d & e & f \end{Bmatrix}
\begin{Bmatrix}  b & f & d \\ h & i & g \end{Bmatrix}
= \sum_l 
\begin{Bmatrix}  c & e & d \\ h & i & l \end{Bmatrix}
\begin{Bmatrix}  b & a & c \\ l & i & g \end{Bmatrix}
\begin{Bmatrix}  a & g & l \\ h & e & f \end{Bmatrix}.
$$

In this paper we will use a slightly modified version of the quantum $6j$-symbols, which we will call \emph{renormalized $6j$-symbols}, that have no poles in zero:

\begin{prop} \label{exist:renormalized:prop}
There is a family of analytic functions
$$\begin{Bmatrix}  a & b & c \\ d & e & f \end{Bmatrix}^R$$
defined on $\overline \matD\setminus S$ and determined by six parameters $a,b,c,d,e,f \in \matN$ which satisfy the orthogonality relations and the Biedenharn-Elliot identities for every $A\in\overline \matD\setminus S$. Moreover:
\begin{align}
\label{symmetry:eqn}
\begin{Bmatrix}  a & b &c \\ d & e & f \end{Bmatrix}^R & =
\begin{Bmatrix}  a & e & f \\ d & b & c \end{Bmatrix}^R, \\
\label{real:eqn}
\begin{Bmatrix} a & b & c \\ d & e & f \end{Bmatrix}^R(A) & \in \matR \ {\rm if}\ A\in \matR\cup i\matR, \\
\label{zero:eqn}
\begin{Bmatrix} a & b & c \\ d & e & f \end{Bmatrix}^R(0) & = \left\{
\begin{array}{rl}
1 & {\rm if }\ c + f = \max\{a+d, b+e\}, \\
0 & {\rm if }\ c + f \neq \max\{a+d,b+e\}
\end{array} \right. \\
\label{Q:eqn}
\left| \begin{Bmatrix}  a & b & c \\ d & e & f \end{Bmatrix}^R(A)\right| & \leqslant |A|^{Q(a,b,c,d,e,f)}K(A) \quad {\rm if} A\in \matD
\end{align}
for some non-negative functions $Q\colon \matN^6 \to \matN$ and a continuous function 
$K\colon \matD\to \matR$ such that for every $a,b,c,d,e,n\in \matN$ the following inequality holds: 
$$\#\big\{f\ \big|\ Q(a,b,c,d,e,f)=n\big\}\leqslant 12.$$
\end{prop}

The renormalized $6j$-symbol is a square root of a rational function (note the analogy with the function $\sqrt{1-z^2}$ from Section \ref{technique:subsection}). We will define the renormalized $6j$-symbols in Section \ref{estimate:section}, where we will prove \eqref{zero:eqn} and \eqref{Q:eqn} by making use of some crucial estimates of Frohman and Kanya-Bartoszy\'nska \cite{FK2}. 

\subsection{Multicurves}
Recall that $X$ is the set of all triangulations on $\Sigma$ with vertices at the marked points and $\calM$ is the set of all multicurves on $\Sigma$;  
we want to introduce an analytic family of cycles 
$$c_A\colon X\times X \to B(\calH)$$
with $\calH = \ell^2(\calM)$. To do so we need to link multicurves, triangulations, and quantum $6j$-symbols.

\begin{defn} 
Let $x$ be a triangulation. A multicurve is said to be in \emph{normal form} with respect to $x$ if it intersects every triangle transversely in arcs joining distinct edges. An \emph{admissible coloring} of $x$ is the assignment of a non-negative integer at each edge of $x$ such that the three numbers $i,j,k$ coloring the three edges of any triangle satisfy the triangle inequalities and have an even sum $i+j+k$.
\end{defn}

A multicurve put in normal form induces an admissible coloring on the triangulation $x$, defined simply by counting the intersections of every edge of $x$ with the multicurve. 

\begin{prop} \label{normal:prop}
Normal form induces a natural bijection between $\calM$ and the set of all admissible colorings on $x$.
\end{prop}
\begin{proof}
Every multicurve can be put in normal form. By the \emph{bigon criterion} \cite{FM} the intersection number between a multicurve in normal form and an edge of the triangulation is their minimum geometric intersection and is hence uniquely determined by the isotopy class of the multicurve.
\end{proof}

When two triangulations $x$ and $x'$ are related by a flip the coordinates of multicurves change in a  very simple manner. We leave the following proposition as an exercise.

\begin{figure}
 \begin{center}
  \includegraphics[width = 7 cm]{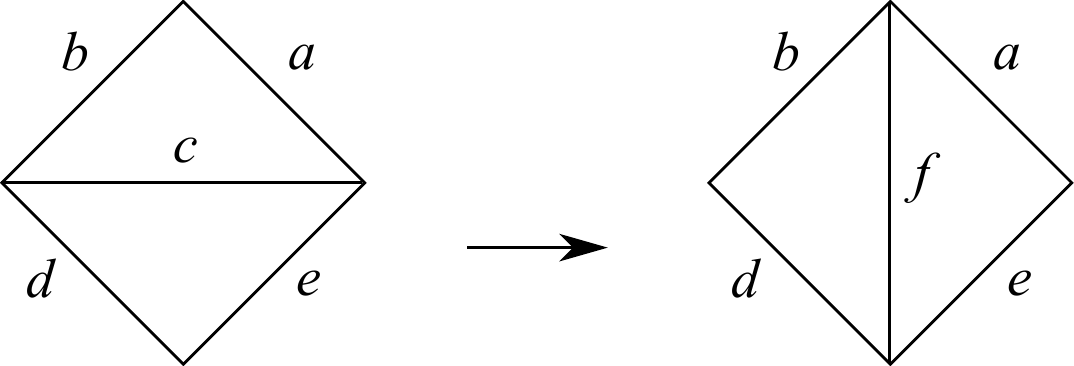}
 \end{center}
 \caption{A flip on a colored triangulation.}
 \label{flip_colored:fig}
\end{figure}

\begin{prop} \label{related:prop}
Let two triangulations be related via a flip as in Fig.~\ref{flip_colored:fig}. The left and right colorings determine the same multicurve provided that
$$c+f = \max \{a+d, b+e\}.$$
\end{prop}

\subsection{The cocycle} \label{cocycle:subsection}
We fix for the rest of this section a complex number $A\in \overline\matD \setminus S$. We can finally define the cocycle $c_A$. As customary with the cocycle technique we first define the operator
$$c_A(y,x)\colon \calH \to \calH$$
on a pair $x$, $y$ of vertices of $X$ connected by an edge, that is on two triangulations related by a flip. The Hilbert space $\calH = \ell^2(\calM)$ has an obvious Hilbert basis $\{\delta_\gamma\}_{\gamma \in \calM}$ and we set 
$$c_A(y,x)(\delta_\gamma) = \lambda_1\delta_{\gamma_1} + \ldots + \lambda_k\delta_{\gamma_k}$$
where the multicurves $\gamma_i$ and the complex coefficients $\lambda_i$ depend on $\gamma$ as follows. Represent $\gamma$ as an admissible coloring on the triangulation $x$ as in Fig.~\ref{flip_colored:fig}-(left). Corresponding to that coloring, there are finitely many colorings on $y$ as in Fig.~\ref{flip_colored:fig}-(right) that differ only by the value $f$ assigned on the new edge: these are finite in number because of the triangle inequality $f\leqslant a+e$. 

We define the multicurves $\gamma_1,\ldots, \gamma_k$ as those corresponding to that finitely many colorings on $x$. For each such $\gamma_i$ we define the corresponding coefficient $\lambda_i$ as the renormalized $6j$-symbol
$$\begin{Bmatrix}  a & b & c \\ d & e & f \end{Bmatrix}^R.$$
Recall that $\check\calH\subset \calH$ consists of all finitely supported functions on $\calM$, \emph{i.e.}~it is the linear span of the orthonormal basis $\{\delta_\gamma\}_{\gamma \in \calM}$. We have just defined a map
$$c_A(y,x)\colon \check\calH \to \check\calH.$$
For a generic pair $x$, $y$ of triangulations we choose as usual a path $x=x_1, \ldots,$ $x_{n-1}, x_n=y$ in the graph $X$ joining $x$ and $y$ and define
$$c_A(y,x) = c_A(x_n,x_{n-1}) \circ \cdots \circ c_A(x_2,x_1).$$
Recall that $S\subset\partial \matD$ is the set of all roots of unity except $\pm 1$ and $\pm i$. 
\begin{lemma} \label{bounded:cor}
The map $c_A(y,x)\colon \check\calH\to \check\calH$ is well-defined for every $A\in \overline \matD \setminus S$. 
\end{lemma}
\begin{proof}
We need to show that choosing a different path connecting $x$ to $y$ we get the same map. By Theorem \ref{pentagon:teo} two sequences differ by a combination of moves (1)-(3). Addition of removal of a pair (1) leaves the map unchanged because of the orthogonality identity of the renormalized $6$-symbols, while the pentagon relation (3) also leaves it unchanged by the Biedenharn-Elliot identity. Invariance under commuting flips (2) is obvious.
\end{proof}
We can finally use the cocycle to define our family of infinite-dimensional representations.
\begin{defn}
Fix a triangulation $x\in X$. For every $A \in \overline \matD \setminus S$ define
$$\rho_A\colon \MCG(\Sigma) \to \GL(\check\calH)$$
by setting $\rho_A(g) = c_A(x,gx) \rho_0(g)$.
\end{defn}
The family mildly depends on the choice of the fixed triangulation $x$, see Remark \ref{dependenceonx:rem} below.

\subsection{Bounded and unitary representations} \label{proofthm1:sub}
We have defined a family $\{\rho_A\}_{A\in \overline \matD\setminus S}$ of representations on the dense subspace $\check \calH\subset\calH$. We prove here the following.

\begin{teo} \label{unitary:teo}
The representation $\rho_A$ is bounded for all $A\in \matD$ hence it extends continuously to $\calH$. In particular if $A\in\matR\cup i\matR$ it is unitary and if $A=0$ it is the multicurve representation. Moreover, for any $v,v'\in \calH$ and any $g\in \MCG (\Sigma)$ the function 
$\langle \rho_A(g)v,v'\rangle$ is analytic in $A\in \matD$.
\end{teo}

We prove Theorem \ref{unitary:teo} as a sequence of lemmas, assuming Proposition \ref{exist:renormalized:prop}.

\begin{lemma} \label{orthogonal:lemma}
The map $c_A(y,x)\colon \check\calH \to \check\calH$ is unitary when $A\in \matR\cup i\matR$. Moreover $c_0(y,x)=\id$.
\end{lemma}
\begin{proof}
We can suppose without loss of generality that $x$ and $y$ are related by a flip as in Fig~\ref{flip_colored:fig}. Proposition \ref{related:prop} shows how the same multicurve is expressed as an admissible coloring on $x$ and $y$, and Proposition \ref{exist:renormalized:prop}-(\ref{zero:eqn}) then implies that $c_0(x,y)=\id$. The symmetry in Proposition \ref{exist:renormalized:prop}-(\ref{symmetry:eqn}) and the orthogonality of the renormalized $6j$-symbol together show that 
$$^tc_A(y,x)c_A(y,x) = \id.$$
When $A\in \matR\cup i \matR$ we have $c_A(y,x)\in \matR$ by Proposition \ref{exist:renormalized:prop}-(\ref{real:eqn}) and therefore $^tc_A(y,x) = c_A^*(y,x)$; hence 
$c_A(y,x)$ is unitary in that case.
\end{proof}

\begin{lemma} \label{bounded:lemma}
We have 
$$\|c_A(y,x)\| < f(A)$$
for some continuous function $f\colon \matD \to \matR_+$.
In particular when $|A|<1$ the map $c_A(y,x)$ is 
bounded and therefore extends to $\calH$. 
\end{lemma}
\begin{proof}
We can suppose that $x$ and $y$ are connected by a flip as in Fig.~\ref{flip_colored:fig}. Let a \emph{partial coloring} $\sigma$ on $x$ be an admissible coloring of all edges of $x$ except the flipped one $c$, which we leave uncolored. A partial coloring can be completed only to finitely many full colorings of $x$ since $c\leqslant a+b$, and these finitely many colorings identify finitely many multicurves $\gamma_1,\ldots, \gamma_k$ and hence a finite-dimensional subspace $V_\sigma\subset\calH$ spanned by $\delta_{\gamma_1},\ldots, \delta_{\gamma_k}$. 

Proposition \ref{related:prop} implies easily that $V_\sigma$ is an invariant subspace of $c_A(y,x)$. We have decomposed $c_A(y,x)$ in an orthogonal sum of finite-dimensional operators acting on
$$\check\calH = \oplus_\sigma V_\sigma$$
The finite-dimensional operator acting on a factor $V_\sigma$ is represented as a matrix
$$M_{ij} = \begin{Bmatrix} a & b & i \\ c & d & j \end{Bmatrix}^R$$
for some fixed integers $a,b,c,d$ determined by $\sigma$. To prove that $c_A(y,x)$ is bounded we will show that the operator norm $\|M\|$ of $M_{ij}$ is bounded by a constant $f(A)$ which depends (continuously) only on $A$ and \emph{not} on the integers $a,b,c,d$. Note that the dimension of $V_\sigma$ depends on $a,b,c,d$ and is not uniformly bounded.

Recall that the $p$-norm of a vector $x\in \matR^n$ is $\|x \|_p = \sqrt[p]{|x_1|^p + \ldots + |x_n|^p}$ and the corresponding operator $p$-norm $\|M\|_p$ of a $n\times n$ matrix $M$ is the infimum of $\frac{\|Mv\|_p}{\|v\|_p}$ on all vectors $v\neq 0$. Of course when $p=\infty$ we set $\|x\|_\infty = \max \{|x_i|\}$ and the following inequality holds for any square matrix:
$$\|M\|_2 \leqslant \sqrt{\|M\|_1\|M\|_\infty}.$$
Therefore in order to prove that our operator norm $\|M\| = \|M\|_2$ is uniformly bounded we will bound both $\|M\|_1$ and $\|M\|_\infty$.
These two norms are easier to estimate since $\|M\|_1$ (resp.~$\|M\|_\infty$) is the maximum value of $\sum_i |M_{ij}|$ (resp.~$\sum_j |M_{ij}|$) among all columns $j$ (resp.~rows $i$) of $M$.

Fix a column $j$. Proposition \ref{exist:renormalized:prop}-(\ref{Q:eqn}) implies that there is a continuous function $K(A)$ such that
$$\sum_{i} |M_{ij}| \leqslant 12(|A| + |A|^2 + |A|^3 + \ldots) K(A) \leqslant 12\frac{K(A)}{1-|A|}$$
and therefore $\|M\|_1 \leqslant 12\frac{K(A)}{1-|A|}$. For $\sum_j |M_{ij}|$ we get the same inequality using the symmetry in Proposition \ref{exist:renormalized:prop}-(\ref{symmetry:eqn}) which transposes $M$ and permutes the values $a,b,c,d$. Therefore $\|M\|_\infty$ is bounded analogously and we get
$$\|M\|_2 \leqslant 12\frac{K(A)}{1-|A|}.$$
This finally implies that
$$\|c_A(y,x)\| \leqslant 12\frac{K(A)}{1-|A|}.$$
\end{proof}

\begin{cor}
For any $g\in \MCG(\Sigma)$ there is a continuous function $f\colon \matD \to \matR_+$ such that
$$\|\rho_A(g)\| < f(A).$$
\end{cor}

The representations $\{\rho_A\}_{A\in \matD}$ are bounded and therefore extend to $\calH$. It remains to prove that they vary analytically:
\begin{lemma}\label{lem:analyticity}
For any $v,v'\in \calH$ and any $g\in \MCG(\Sigma)$ the function 
$\langle \rho_A(g)v,v\rangle$ is analytic in $A\in \matD$.
\end{lemma}
\begin{proof}
Observe first that the statement is true for $v,v'\in \check \calH$: indeed $\langle \rho_A(g)v,v'\rangle$ is (by construction) a finite linear combination of products of renormalized $6j$-symbols each of which depends analytically on $A\in \matD$.  Let now $v,v'\in \calH$ and $K\subset \matD$ be a compact set; we will prove now analyticity on $K$.
Let $(v_n)_{n\in \matN}$ and $(v'_n)_{n\in \matN}$ be sequences of vectors in $\check \calH$ tending to $v,v'$ respectively, and let $C=\max\{||\rho_A(g)||\  |A\in K\}$. 
For any $\epsilon>0$ there exists $n_0 \in \matN$ such that $||v-v_n||<\epsilon, ||v'-v'_n||<\epsilon,\ \forall n>n_0$; so the following two inequalities hold:
$$| \langle\rho_A(g)v,v'\rangle-\langle \rho_A(g)v_n,v'_n\rangle|\leqslant C(||v'||+||v|| + \epsilon)\epsilon,\ \forall A\in K,$$
$$| \langle\rho_A(g)v_n,v'_n\rangle|\leqslant C(||v||+\epsilon)(||v'||+\epsilon),\ \forall A\in K.$$
Then $\langle\rho_A(g)v,v'\rangle$ is a uniform limit on $K$ of analytic functions and is hence also analytic.
\end{proof}
Theorem \ref{unitary:teo} is now an easy corollary of the previous lemmas. We end this section with a couple of remarks.

\begin{rem} \label{dependenceonx:rem}
The analytic family of representations depends on a choice of a triangulation $x$ for $\Sigma$. This dependence is however very mild: two families corresponding to different triangulations $x$ and $y$ are related by a \emph{canonical} analytic family of operators $\psi_A\colon \check\calH \to \check\calH$ that are bounded on $|A|<1$, unitary if $A\in \matR\cup i \matR$, and with $\psi_0 = \id$. In particular the unitary representations $\rho_A$ on the real and imaginary axis are uniquely determined up to isometries. This is a general consequence of the cocycle technique.
\end{rem}

\begin{rem}
The proof of Lemma \ref{bounded:lemma} shows that when $x$ and $y$ are connected by a flip the map $c_A(y,x)$ decomposes into (infinitely many) \emph{finite}-dimensional uniformly bounded operators. Note the analogy with Valette's cocycle in Section \ref{technique:subsection}, where $c_z(y,x)$ also decomposes into finite-dimensional operators: a crucial difference is that Valette's operator acts non-trivially only on a single plane while $c_A(y,x)$ may act non-trivially on subspaces of arbitrarily big dimension (although in a uniformly bounded way).
\end{rem}

\section{The Kauffman bracket} \label{Kauffman:section}

We introduce here the quantum $6j$-symbols via the Kaufmann bracket following \cite{L_book}.  All the material here is standard: the renormalized $6j$-symbols will be introduced only in the next section.

\subsection{The Kauffman module} \label{module:subsection}
\begin{figure}
 \begin{center}
  \includegraphics[width = 5.5 cm]{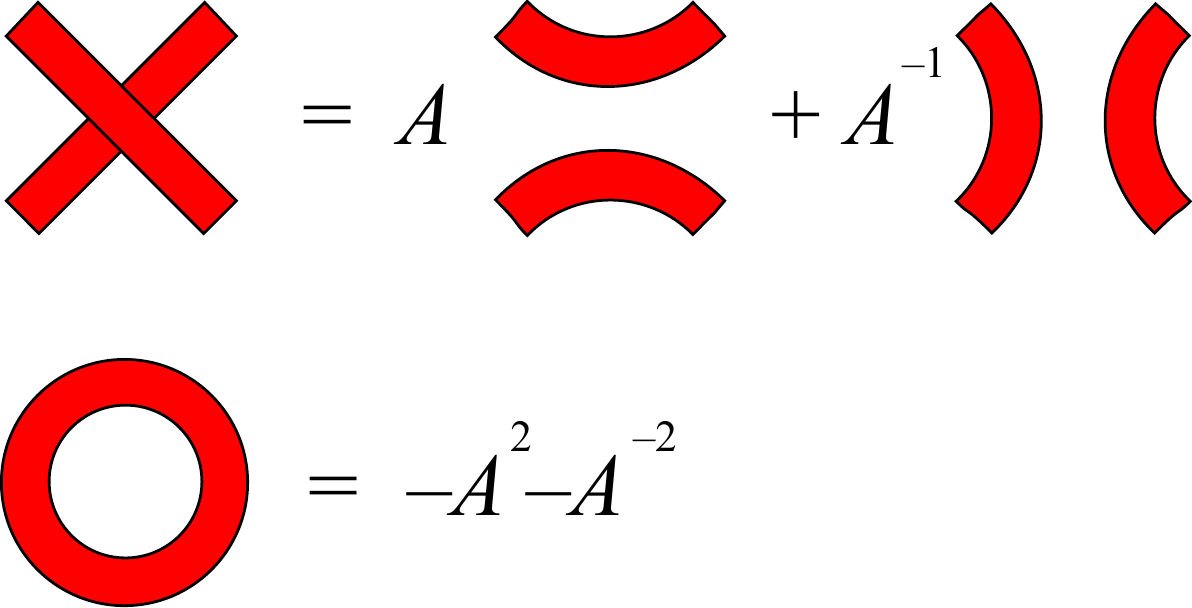}
 \end{center}
 \caption{The Kauffman bracket relations.}
 \label{Kauffman_bracket:fig}
\end{figure}
Let $A\in\matC^*$ be a non-zero complex number. The \emph{Kauffman skein module} of an oriented $3$-manifold $M$ is the $\matC$-vector space $K_A(M)$ generated by all isotopy classes of framed links in $M$, modulo the usual Kauffman bracket relations shown in Fig.~\ref{Kauffman_bracket:fig}. An element of $K_A(M)$ is called a \emph{skein}.

\begin{prop}[Kauffman \cite{Ka}]  \label{canonical:prop}
The space $K_A(S^3)$ is one dimensional and spanned by the class of the empty link.
\end{prop}

In other words every skein in $K_A(S^3)$ is equivalent to $k\cdot \emptyset$ for a well-defined complex number $k$, which is the \emph{evaluation} of the skein. 

\subsection{The Jones-Wenzl projectors}\label{sub:jw}
For the rest of this section we will suppose that $A\not \in S$, that is $A$ is not a root of unity except $\pm 1$ and $\pm i$. We define the  \emph{quantum integers}
$$[n] = \frac{A^{2n}- A^{-2n}}{A^2 - A^{-2}} = A^{-2n+2} + A^{-2n+6} + \ldots + A^{2n-6} + A^{2n-2}$$
and note that $[n]$ is a Laurent polynomial in $A$ whose zeroes are contained in the set $S$. Therefore our assumption $A\not\in S$ guarantees that $[n]\neq 0$.
The existence of $[n]^{-1}$ allows us to define some particularly useful skeins, the \emph{Jones-Wenzl projectors}, as follows.

\begin{figure}
 \begin{center}
  \includegraphics[width = 8 cm]{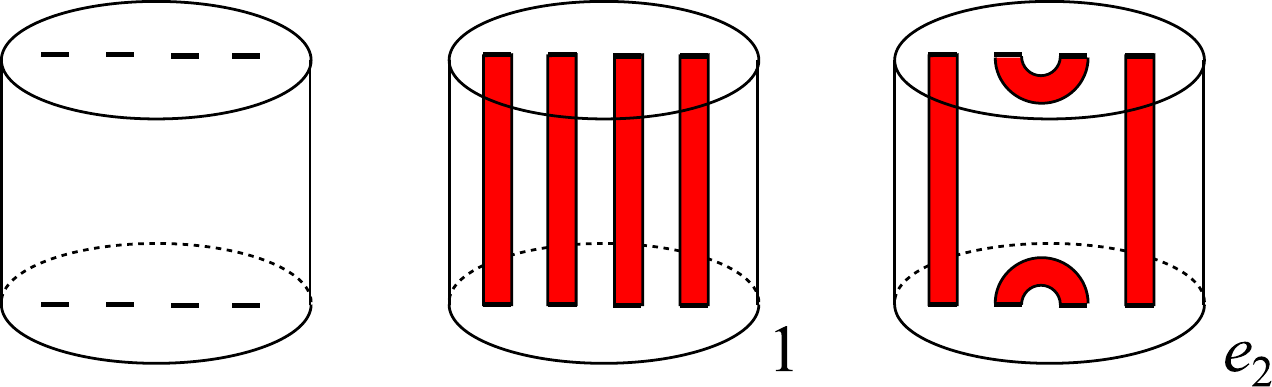}
 \end{center}
 \caption{The Kauffman bracket $K_A(M)$ of the cylinder witn $2n$ marked points: here $n=4$ (left). The space $K_A(M)$ is an algebra generated by the elements $1, e_1,\ldots, e_{n-1}$: here we draw $e_2$ (right).}
 \label{cylinder:fig}
\end{figure}
There is a natural boundary version of the skein module. Let $M$ be an oriented manifold with boundary and $\partial M$ contain some disjoint oriented segments as in Fig.~\ref{cylinder:fig}-(left). The skein module $K_A(M)$ is then defined as above by taking framed links and rectangles intersecting $\partial M$ in those segments.

For instance, we may take $M$ to be a cylinder with $2n$ segments as in Fig.~\ref{cylinder:fig}-(left). Cylinders can be stacked over each other, and hence $K_A(M)$ has a natural algebra structure (called the \emph{Temperly-Lieb algebra}) whose multiplicative identity element is the skein $1$ shown in Fig.~\ref{cylinder:fig}-(centre). We define the elements $e_1,\ldots, e_{n-1}$ as suggested in Fig.~\ref{cylinder:fig}-(right): it is easy to prove that $K_A(M)$ is generated as an algebra by the elements $1, e_1, \ldots, e_{n-1}$.

The $n$-th \emph{Jones-Wenzl projector} $f_n$ is a skein in the cylinder defined inductively as in Fig.~\ref{JW:fig}. It satisfies the following remarkable properties \cite[Lemma 2]{Li}:
$$f_n\circ f_n = f_n, \quad f_n\circ e_i = e_i\circ f_n = 0 \quad \forall i.$$
So $f_n$ is a projector which  ``kills'' the skeins with short returns like the $e_i$'s. Let $I_n$ be the ideal generated by $e_1,\ldots, e_{n-1}$: it follows from the recursive definition that 
$$f_n = 1 + i_n \quad {\rm for\ some\ } i_n \in I_n.$$

\begin{figure}
 \begin{center}
   \includegraphics[width = 10 cm]{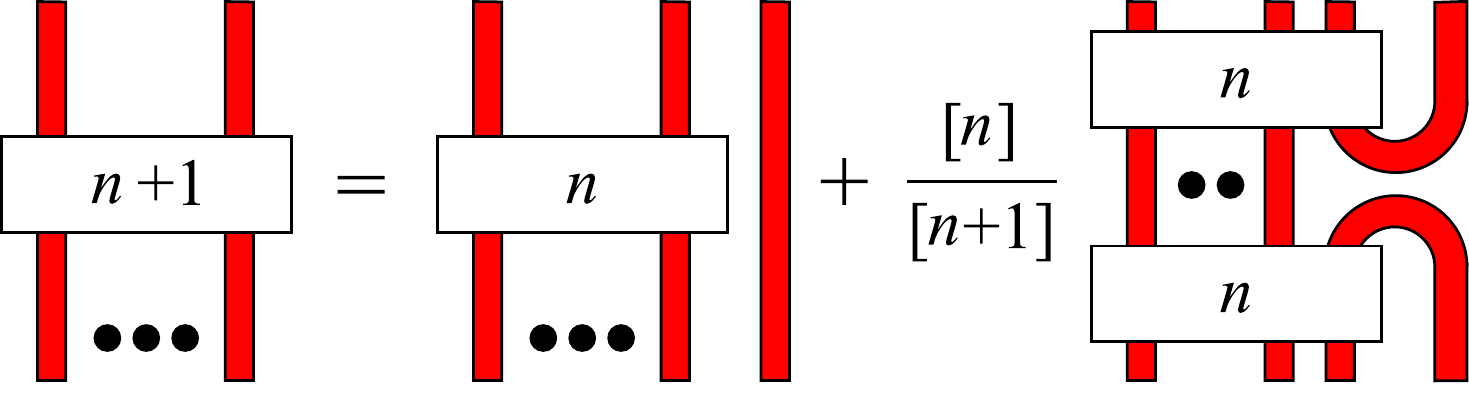}
   \end{center}
 \caption{The $(n+1)^{th}$ Jones-Wenzl projector is defined recursively with this formula.}
 \label{JW:fig}
\end{figure}
 
\subsection{Ribbon graphs} \label{ribbon:subsection}
The Jones-Wenzl projectors are used as building blocks to construct skeins in a simple combinatorial way. A \emph{ribbon graph} $Y\subset M$ is a 3-valent graph with a two-dimensional oriented thickening considered up to isotopy (it is the natural generalization of a framed link). An \emph{admissible coloring} of $Y$ is the assignment of a natural number (a \emph{color}) at each edge of $Y$ such that the three numbers $i,j,k$ coloring the three edges incident to a vertex satisfy the triangle inequalities, and their sum $i+j+k$ is even. These admissibility requirements allow to associate uniquely to $Y$ a skein as suggested in Fig.~\ref{ribbon:fig}. A framed link can be viewed as a colored ribbon graph without vertices whose components are colored with 1.

\begin{figure}
 \begin{center}
  \includegraphics[width = 9 cm]{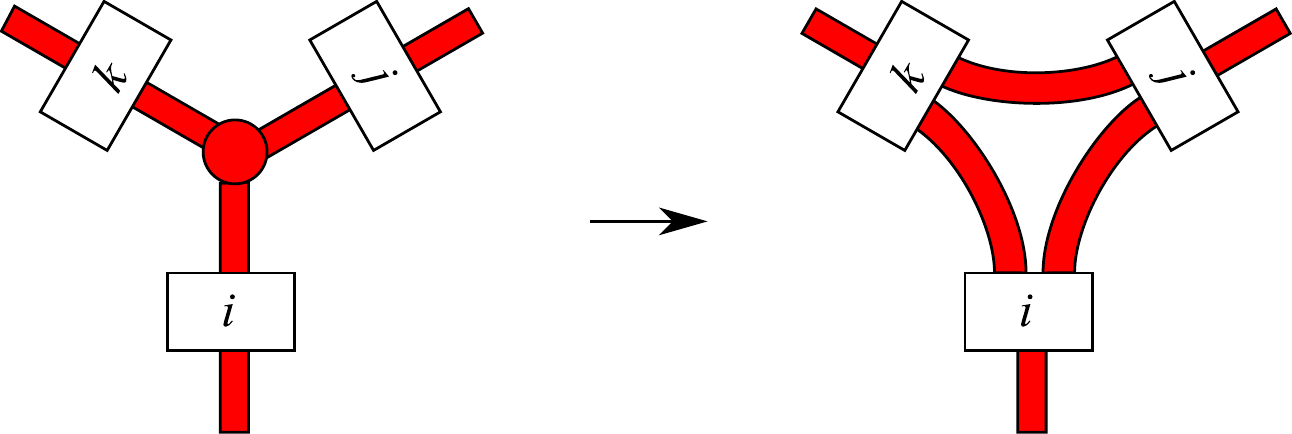}
 \end{center}
 \caption{A colored ribbon graph determines a skein: replace every edge with a projector, and connect them at every vertex via non intersecting strands contained in the depicted bands. For instance there are exactly $i+j-k$ bands connecting the projectors $i$ and $k$.}
 \label{ribbon:fig}
\end{figure}

\begin{figure}
 \begin{center}
  \includegraphics[width = 10 cm]{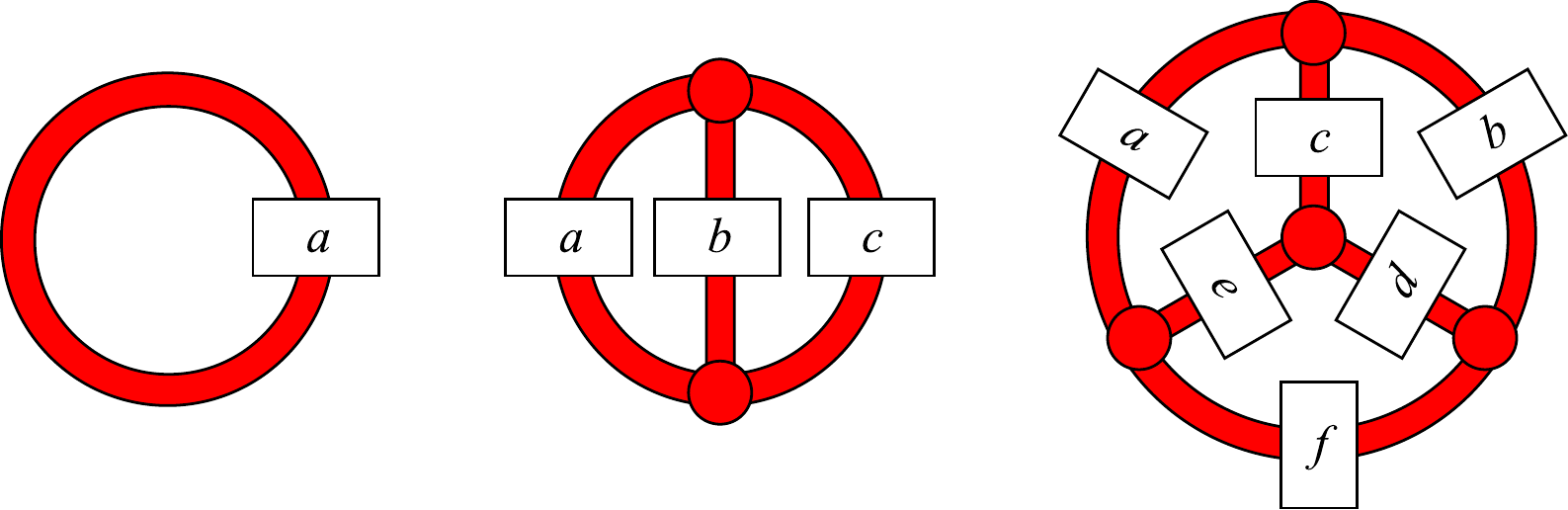}
 \end{center}
 \caption{Three important planar ribbon graphs in $S^3$.}
 \label{graphs:fig}
\end{figure}

Three basic planar ribbon graphs in $S^3$ are shown in Fig.~\ref{graphs:fig}. Since $K_A(S^3) =\matC$, every such ribbon graph provides a complex number which can be expressed as a rational function in the variable $A$. These functions are typically expressed in terms of the quantum integers $[n]$. We take from \cite{KL} the evaluations of the graphs $\cerchio$, $\teta$, and $\tetra$. 
We recall the usual factorial notation 
$$[n]! = [1]\cdots [n]$$ 
with the convention $[0]! = 1$. Similarly one defines multinomial coefficients replacing standard factorials with quantum factorials:
$$\begin{bmatrix} n \\ n_1, \ldots n_k \end{bmatrix} = \frac{[n]!}{[n_1]!\cdots [n_k]!}.$$
When using multinomial coefficients we always suppose that $n = n_1+ \ldots + n_k$. The evaluations of $\cerchio$, $\teta$ and $\tetra$ are:
\begin{align*}
\cerchio_a  & = (-1)^{a} [a+1], \\
\teta_{a,b,c}  & = (-1)^{\frac{a+b+c}2} \frac{\left[\frac{a+b+c}2+1\right]! \left[\frac{a+b-c}2\right]! \left[\frac{b+c-a}2\right]! \left[\frac{c+a-b}2\right]!}{[a]![b]![c]!},  \\ 
\raisebox{-0.5cm}{\tetracolored } & = \frac{ \prod_{i=1}^3\prod_{j=1}^4 [\Box_i-\triangle_j]}{[a]![b]![c]![d]![e]![f]!} \times \\
& \quad \sum_{z = \max \triangle_j }^{\min \Box_i}\!\!\! (-1)^z 
\begin{bmatrix} z+1 \\
z-\triangle_1, z-\triangle_2, z-\triangle_3, z-\triangle_4, \Box_1-z, \Box_2-z, \Box_3-z, 1 \end{bmatrix}.
\end{align*}

In the latter equality, triangles and squares are defined as follows:
\begin{align*}
\triangle_1 = \frac{a+b+c}{2},\ \triangle_2 = \frac{a+e+f}{2},\ \triangle_3 =\frac{ d+b+f}{2},\ \triangle_4 = \frac{d+e+c}{2},\\
\Box_1 = \frac{a+b+d+e}{2},\ \Box_2 = \frac{a+c+d+f}{2},\ \Box_3 = \frac{b+c+e+f}{2}.
\end{align*}

The formula for the planar tetrahedron was first proved by Masbaum and Vogel \cite{MV}. We note that the evaluations are rational functions with poles in $S\cup\{0,\infty\}$. It is actually easy to check from the definitions that the evaluation of any ribbon graph in $S^3$ is a rational function with poles contained in $S\cup \{0,\infty\}$.

\subsection{The skein module of the disk}
We now follow \cite[Section 14]{L_book}. Recall that we suppose in this section that $A\in\matC^* \setminus S$, that is $A\in \matC^*$ is not a root of unity except $\pm 1$ and $\pm i$.

Let $a_1, \ldots, a_n$ be non-negative integers with even sum and consider the skein module $K_A(D^2) = K_A(D^2\times [0,1])$ of the 3-disk $D^2\times [0,1]$ with $a_1+\ldots + a_n$ marked points in the boundary $\partial D^2 \times \frac 12$. The points are partitioned into $n$ sets of consecutive points of cardinality $a_1,\ldots, a_n$, and by inserting the Jones-Wentzl projectors $f_{a_1}, \ldots, f_{a_n}$ at these sets of points we define an idempotent endomorphism (a projection) of $K_A(D)$. 
\begin{defn}
Let $T_{a_1,\ldots, a_n}$ be the image of this projection.
\end{defn}
Let now $Y\subset D$ be a tree with vertices of valence 1 and 3, intersecting $\partial D$ precisely in its 1-valent vertices. An admissible colouring $\sigma$ of $Y$ which extends the boundary colourings $a_1,\ldots, a_n$ determines an element of $T_{a_1,\ldots, a_n}$. Note that when $n=1$ such a tree does not exists, when $n=2$ it is a line which can be admissibly coloured only when $a_1=a_2$, and when $n=3$ the tree is $Y$-shaped and can be admissibly coloured (in a unique way) if and only if $a_1,a_2,a_3$ form an admissible triple. When $n=4$ the tree is a $H$-shaped graph which can be admissibly coloured in at most finitely many ways. We denote by $Y_\sigma$ the ribbon graph $Y$ equipped with an admissible colouring $\sigma$.

\begin{prop} \label{T:prop}
The elements $\{Y_\sigma\}$ where $\sigma$ varies among all admissible colourings of $Y$ extending the boundary colorings $a_1,\ldots, a_n$ form a basis of $T_{a_1,\ldots, a_n}$.
\end{prop}
\begin{proof}
This is proved in \cite[Section 14]{L_book} for $n\leqslant 4$ and the proof easily extends to any $n$.
\end{proof}

When $n=4$ there are two possible graphs $Y$ and they give rise to different bases of $T_{a_1,a_2,a_3,a_4}$. The change of basis is easily seen \cite[Chapter 14]{L_book} to be as depicted in Fig.~\ref{whitehead2:fig}. 

\begin{defn}\label{6j:defn}
The coefficient of the $f^{th}$-summand in the change of basis from Fig.~\ref{whitehead2:fig} is the \emph{quantum $6j$-symbol}, denoted:
$$\begin{Bmatrix}  a & b & c \\ d & e & f \end{Bmatrix}.$$ 
\end{defn}

The relation between skeins in Fig.~\ref{whitehead2:fig} is also called a \emph{Whitehead move}. The special case $c=0$ gives rise to the \emph{fusion rule} shown in Fig.~\ref{fusion:fig}. (Recall that from the very definition of the skein associated to a colored ribbon graph, a $0$-colored edge can be deleted without changing the associated skein.)
In both identities the sum ranges over all the finitely many values giving an admissible coloring to the right-most graph.
\begin{figure}
 \begin{center}
  \includegraphics[width = 12.5 cm]{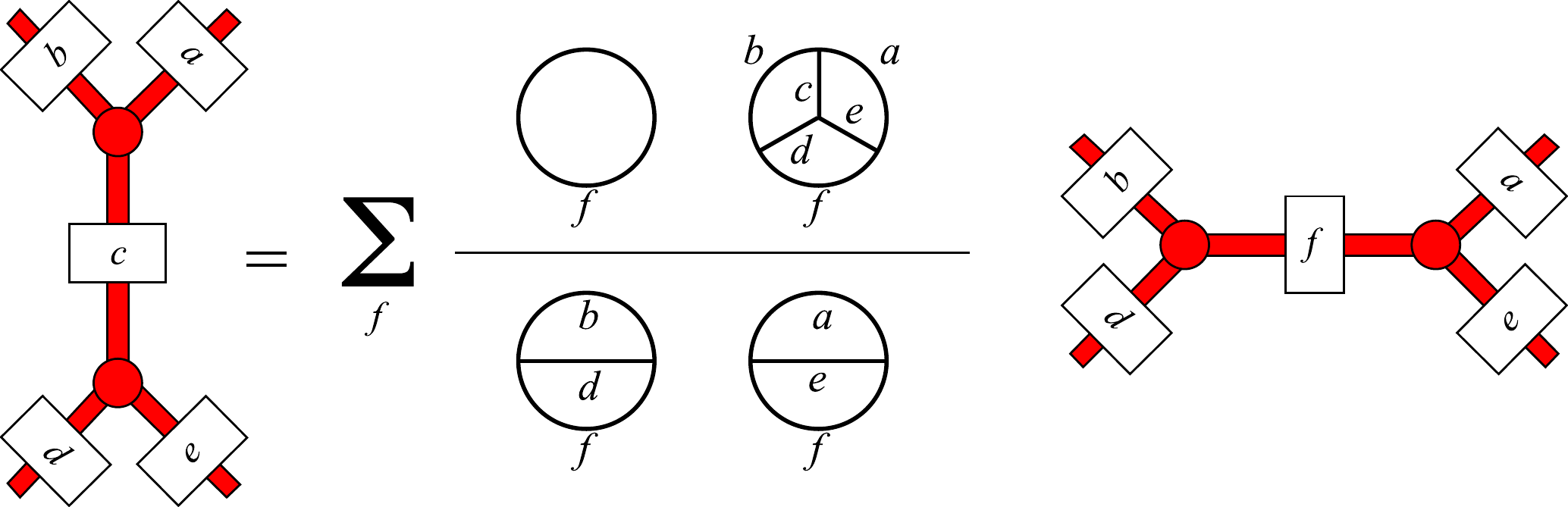}
 \end{center}
 \caption{The \emph{Whitehead move}: the summation is over all the admissible colors (and is hence finite).}
 \label{whitehead2:fig}
\end{figure}

\begin{figure}
 \begin{center}
  \includegraphics[width =9 cm]{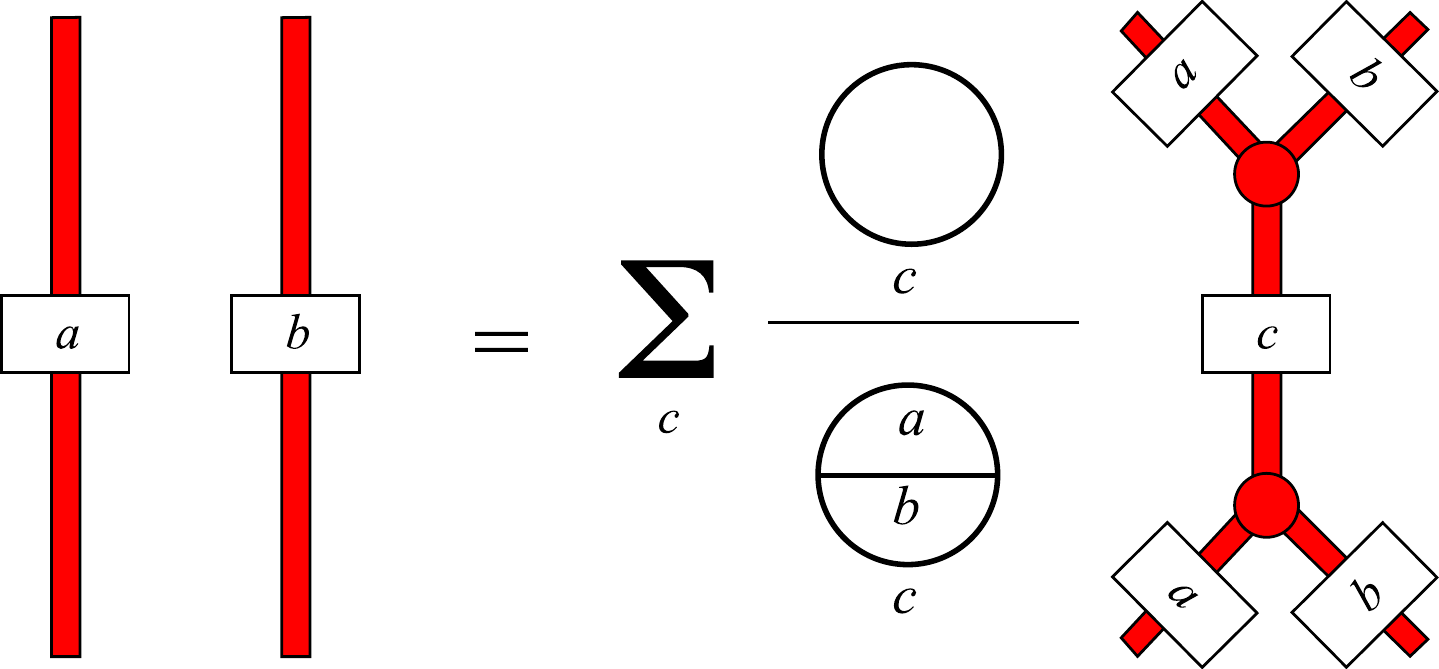}
 \end{center}
 \caption{The \emph{fusion rule}: it is a special case of the Whitehead move.}
 \label{fusion:fig}
\end{figure}

The orthogonal relation and Biedenharn-Elliot identity follow immediately.
\begin{prop} \label{identities:prop}
Let $A\in \matC^* \setminus S$. The quantum $6j$-symbols fulfill the orthogonal relation and the Biedenharn-Elliot identity.
\end{prop}
\begin{proof}
The orthogonal relation follows by changing basis in $T_{a_1,a_2,a_3,a_4}$ twice. The Biedenharn-Elliot identity follows by changing basis five times in $T_{a_1,a_2,a_3,a_4,a_5}$ following the pentagon move in Fig.~\ref{pentagon:fig}.
\end{proof}

\subsection{The reduced skein module} \label{reduced:subsection}
We now consider the case $A\in S$, that is $A$ is a root of unity distinct from $\pm 1$ and $\pm i$. Recall that $r=r(A)\geqslant 2$ is the smallest integer such that $A^{4r}=1$. For such value of $A$ the quantum integer $[n]$ is non-zero for all $n<r$ but $[r]=0$, and hence the Jones-Wenzl projectors $f_1, \ldots,f_{r-1}$ are defined whereas $f_r$ is not, see Fig.~\ref{JW:fig}. Therefore ribbon graphs are defined only when all colourings are smaller or equal than $r-1$. 

Our main goal is to recover the orthogonal relation and the Biedenharn-Elliot identities when $A\in S$: to do so we will use the \emph{reduced} skein module, which was introduced in an unpublished paper by Justin Roberts.

\begin{defn}
The \emph{reduced skein} $K^{\rm red}_A(M)$ of a 3-manifold $M$ is the quotient of $K_A(M)$ by the relations that kill every skein containing a portion as in Fig.~\ref{reduced:fig}. 
\end{defn}

\begin{figure}
 \begin{center}
  \includegraphics[width = 6 cm]{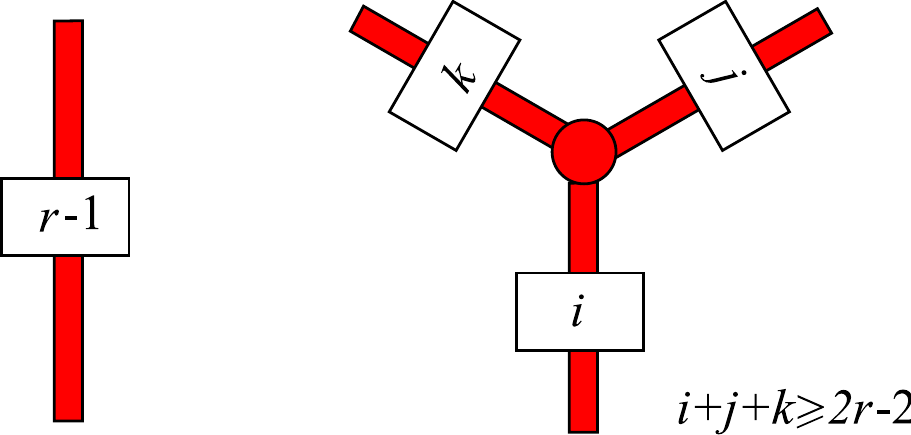}
 \end{center}
 \caption{The reduced skein vector space $K^{\rm red}_A(M)$ is constructed by quotienting $K_A(S)$ by the span of the elements containing one of these two skeins. Concerning the right triple $(i,j,k)$, note that it is defined only when $i,j,k\leqslant r-1$, and that we quotient only by the three-uples $(i,j,k)$ with $i+j+k \geqslant 2r-2$.}
 \label{reduced:fig}
\end{figure}

The crucial point here is that by killing the skeins in Fig.~\ref{reduced:fig} we do not affect the skein module of $S^3$: indeed every skein in $S^3$ containing one of the portions in Fig.~\ref{reduced:fig} is already zero \cite[Lemma 14.7]{L_book}, hence $K_A^{\rm red}(S^3) = K_A(S^3) = \matC$. This is however not true for a general 3-manifold.

Let $a_1,\ldots, a_n \leqslant r-1$ be non-negative integers smaller or equal than $r-1$ with even sum. We define as in the previous section the reduced skein module $K_A^{\rm red} (D^2) = K_A^{\rm red} (D^2\times [0,1])$ of the 3-disc with $a_1+\ldots + a_n$ marked points, its subspace $T_{a_1,\ldots, a_n}$, and a tree $Y\subset D$. 

An admissible colouring of $Y$ is \emph{$r$-admissible} if $a+b+c \leqslant 2(r-2)$ at every 3-valent vertex of $Y$, whose incident edges are coloured as $a$, $b$, and $c$. (When $Y$ is a single edge we also require that its colour is $\leqslant r-2$.)

\begin{prop}
The elements $\{Y_\sigma\}$ where $\sigma$ varies among all $r$-admissible colourings of $Y$ extending the boundary colorings $a_1,\ldots, a_n$ form a basis of $T_{a_1,\ldots, a_n}$.
\end{prop}
\begin{proof}
This is stated in \cite[Lemmas 14,7 and 14.10]{L_book} with some variations: there $n\leqslant 4$, the complex number $A$ is a primitive $(4r)^{\rm th}$ root of unity, and the author defines $T_{a_1,\ldots, a_n}$ as ``maps to the outside''. The proof as stated there also works in our context and easily extends to any $n$. (The hypothesis that $A$ is a primitive $(4r)^{\rm th}$ root of unity is indeed necessary to construct invariants of 3-manifolds and it appears in Lickorish' book at \cite[Lemma 13.7]{L_book}: it is however not needed to prove this particular result, as one can easily check.)
\end{proof}

As in the previous section, there are two possible graphs for $T_{a_1,a_2,a_3,a_4}$ and they give rise to different bases. The change of basis is again as in Fig.~\ref{whitehead2:fig} and the quantum $6j$-symbol are defined analogously. The only difference is that in the summation of Fig.~\ref{whitehead2:fig} the integer $f$ ranges among all $r$-admissible (not ony admissible) colors. The same proof of Proposition \ref{identities:prop} shows the following.

\begin{prop} \label{identities:red:prop}
Let $A\in S$. The quantum $6j$-symbols fulfill the orthogonal relation and the Biedenharn-Elliot identity.
\end{prop}

\subsection{Conclusion} \label{conclusion:subsection}
Summing up, the quantum $6j$-symbol 
\begin{equation}\label{eq:6j}
\begin{Bmatrix}  a & b &c \\ d & e & f \end{Bmatrix}=\frac{ \raisebox{-0.5cm}{\tetracolored }\ \cerchio_c}{\teta_{a,e,f}\teta_{d,b,f}}\end{equation}
is defined for all $A\in \matC^*\setminus F$ for some finite set $F\subset S$ which depends on the parameters $a,b,c,d,e,f$ as follows. If 
$$M = \max\{a+b+c, \ c+d+e, \ a+e+f, \ b+d+f\} $$
then $F$ consists of all $A\in S$ such that $2(r(A)-2) < M$.
The right-hand of (\ref{eq:6j}) shows that the quantum $6j$-symbol is a rational function of $A$, with poles contained in $F \cup \{0, \infty\}$.  

Quantum $6j$-symbols satisfy the orthogonal relation and the Biedenharn-Elliot identity for every $A\in \matC^*$ (where sums must be taken only on admissible or $r$-admissible colorings, depending on whether $A\not \in S$ or $A \in S$).

\section{The renormalized $6j$-symbols and their properties} \label{estimate:section}
In this section we introduce and study the renormalized $6j$-symbols. We will need some crucial estimates of the quantum $6j$-symbols proved by Frohman and Kanya-Bartoszynska when $0<A<1$ is a real number \cite{FK2}. With small modifications, we will extend their proof to apply also to any complex number $A\in\matD^*$.

\subsection{Renormalized and unitary $6j$-symbols} \label{renormalized:subsection}
We introduce a different normalization of the $6j$-symbols and estimate its modulus on the unit disc.

\begin{prop} \label{planar:cor}
The evaluations of $\cerchio_a$ and $\teta_{a,b,c}$ expand in $A=0$ as
\begin{align*}
\cerchio_a & = (-1)^{a} A^{-2a} + o(A^{-2a}), \\
\teta_{a,b,c} & = (-1)^{\frac{a+b+c}2}A^{-(a+b+c)} + o(A^{-(a+b+c)}).
\end{align*}
\end{prop}
\begin{proof} 
For the first statement recall that 
$$\cerchio_a=(-1)^{a}[a+1]=(-1)^a(A^{-2a}+A^{-2a+4}+\ldots +A^{2a}).$$
The second statement is obtained by plugging $[n]!=A^{-n(n-1)}+o(A^{-n(n-1)})$ into the formula for $\teta_{a,b,c}$, thus getting
\begin{align*}
\teta_{a,b,c} & = (-1)^{\frac{a+b+c}{2}}\frac{[\frac{a+b+c}{2}+1]![\frac{a+b-c}{2}]![\frac{a-b+c}{2}]![\frac{-a+b+c}{2}]!}{[a]![b]![c]!} \\
& =(-1)^{\frac{a+b+c}{2}}A^{-(a+b+c)}+o(A^{-(a+b+c)}).
\end{align*}
\end{proof}
The evaluations of $\cerchio_a$ and $\teta_{a,b,c}$ may have zeroes and poles only in $F\cup \{0,\infty\}$ where $F\subset S$ is the finite set consisting of all values $A$ such that $r(A)-2 < a$ and $2(r(A)-2) < a+b+c$ respectively. Therefore we can take their square roots and define the analytic functions
$$\sqrt{\cerchio_a}, \quad \sqrt{\teta_{a,b,c}}$$
on the domain $\overline\matD^* \setminus F$ by requiring that they behave respectively as $i^aA^{-a}$ and $i^{\frac{a+b+c}2}A^{-\frac{(a+b+c)}{2}}$ when $A\to 0$. (Here $\overline \matD^* = \overline \matD \setminus \{0\}$ as usual.) The finite set $F$ of course depends on the parameters $a,b,c$.

Recall from Section \ref{conclusion:subsection} that the quantum $6j$-symbol is also a rational function with poles only in $F\cup \{0,\infty\}$ where $F \subset S$ is a finite set that depends explicitly on the parameters $a,b,c,d,e,f$.

\begin{defn}[Renormalized and unitary symbols] \label{renorm:defn}
The \emph{renormalized $6j$-symbol} is the following analytic function on $\overline\matD^*\setminus F$:
\begin{equation}\label{eq:symm6j}
\begin{Bmatrix}  a & b &c \\ d & e & f \end{Bmatrix}^R=\frac{ \raisebox{-0.5cm}{\tetracolored }\sqrt{\cerchio_c\cerchio_f}}{\sqrt{\teta_{a,b,c}\teta_{a,e,f}\teta_{d,b,f}\teta_{d,e,c}}}\end{equation}
while the \emph{unitary $6j$-symbol}, introduced first in \cite{KR}, is the following:
\begin{equation}\label{eq:unit6j}
\begin{Bmatrix}  a & b &c \\ d & e & f \end{Bmatrix}^U=\frac{ \raisebox{-0.5cm}{\tetracolored }}{\sqrt{\teta_{a,b,c}\teta_{a,e,f}\teta_{d,b,f}\teta_{d,e,c}}}.\end{equation}
Both functions are defined and analytic on $\overline \matD^* \setminus F$. The finite set $F\subset S$ consists of all $A\in S$ such that $2(r(A)-2)<M$ where
$$M = \max\{a+b+c, \ c+d+e, \ a+e+f, \ b+d+f\}. $$
\end{defn}

The unitary symbol has all the symmetries of the tetrahedron and was investigated in \cite{FK2}. The renormalized $6j$-symbol has all the symmetries of the tetrahedron that preserve the pair of edges $\{c,f\}$ as a set. We can immediately prove half of Proposition \ref{exist:renormalized:prop}:

\begin{lemma}\label{sign:lem}
The renormalized $6j$-symbols satisfy the orthogonality relation and the Biedenharn-Elliot identity for all $A\in \matD^*$. Moreover:
\begin{align*}
\begin{Bmatrix}  a & b &c \\ d & e & f \end{Bmatrix}^R & =
\begin{Bmatrix}  a & e & f \\ d & b & c \end{Bmatrix}^R, \\
\begin{Bmatrix} a & b & c \\ d & e & f \end{Bmatrix}^R(A) & \in \matR \ {\rm if}\ A\in \matR^*\cup i\matR^*.
\end{align*}
\end{lemma}
\begin{proof}
It is easy to check that the renormalized $6j$-symbols satisfy the orthogonality and Biedenharn-Elliot identities because the quantum $6j$-symbols do, as shown by Proposition \ref{identities:prop}.

The first equality arises as a symmetry of the tetrahedron that fixes the pair $\{c,f\}$. We turn to the second equality and suppose $A\in \matR^*\cup i\matR^*$. For any $n\in \matN$ the quantum integer $[n]\in \matR$ is a real number with sign $(-1)^{n-1}$ if $A\in i\matR$ and $+1$ if $A\in \matR$. 
Hence looking at the formulas defining the evaluations of $\cerchio_a, \teta_{a,b,c}$ and $\raisebox{-0.6 cm}{\tetracolored}$ one sees immediately that their values are real. 

So to check the statement it is sufficient to show that the square of the renormalized $6j$-symbol is a positive real number. Clearly the square of \raisebox{-0.5cm}{\tetracolored } is real positive; the remaining terms are easily seen to have positive sign. Indeed if $A\in i\matR$ one checks directly that $\teta_{a,b,c}>0$ and $\cerchio_a>0$. If $A\in \matR$  
then the sign of $\teta_{a,b,c}$ is $(-1)^{\frac{a+b+c}{2}}$ so that by multiplying all these signs one gets $(-1)^{a+d+e+b}=1$ because $a+b+c$ and $c+e+d$ are even.
\end{proof}

\subsection{Estimates}
We will use an important estimate elaborated by Frohman and Kanya-Bartoszynska \cite{FK2}. We slightly modify their proof to apply for non-real values of $A$. A unitary $6j$-symbol corresponds to a colored tetrahedron: 
$$\tetracolored$$
Let $C_1$, $C_2$, $C_3$ be the sums of the colors of opposite edges of the tetrahedron, \emph{i.e.}~the numbers $a+d$, $b+e$, $c+f$, ordered so that
$$C_1\geqslant C_2 \geqslant C_3.$$
We will need the following well known:
\begin{prop}[Euler function \cite{Ap}, Section 14.3]\label{prop:euler}
Let $(t;t)_n:=\prod_{i=1}^n (1-t^i),\ n\in \mathbb{N}\cup \{\infty\}$; then $\lim_{n\to \infty} (t;t)_n=(t;t)_\infty$ uniformly on compact sets in $\matD$. Moreover $(t;t)_\infty$ is a holomorphic nowhere-vanishing function on $\matD$.
\end{prop}
\begin{lemma}[Estimate the $6j$-symbols] \label{FK:lemma}
There is a continuous function $K\colon \matD \to \matR_+$
such that the following inequality holds for all $A\in\matD^*$: 
$$\Big| \begin{Bmatrix}  a & b & c \\ d & e & f \end{Bmatrix}^U\Big| \leqslant |A|^{\frac 12(C_1-C_2)(C_1-C_3)+C_1}K(A).$$
\end{lemma}
\begin{proof}
We follow \cite[Proposition 6]{FK2}.
Let $A^4=t$ and set $(t;t)_n=\prod_{j=1}^{n}(1-t^j)$ and $(t;t)_\infty=\lim_{n\to \infty}(t,t)_n$. The only facts we need about the Euler function $(t,t)_\infty $ 
are summarized in Proposition \ref{prop:euler}.
Observe that $[n]=t^{\frac{-n+1}{2}}\frac{1-t^{n}}{1-t}$ and so 
$$[n]!=t^{-\frac{n(n-1)}{4}}\frac{(t,t)_n}{(1-t)^n}.$$
To warm-up we first prove that there exist two positive real-valued functions $T_1$, $T_2\colon \matD \to \matR_+$ such that the following inequalities hold for all $t\in\matD$:
$$ |t|^{-\frac{a+b+c}{8}}T_1(t)\leqslant \left|\sqrt{\teta_{a,b,c}}\right|\leqslant |t|^{-\frac{a+b+c}{8}} T_2(t).
$$
Indeed from the evaluation of $\teta_{a,b,c}$ we get:
$$\teta_{a,b,c}= (-1)^{\frac{a+b+c}{2}}t^{-\frac{a+b+c}{4}}\frac{(t,t)_{\frac{a+b+c}{2}+1}(t,t)_{\frac{a+b-c}{2}} (t,t)_{\frac{a-b+c}{2}} (t,t)_{\frac{-a+b+c}{2}}} {(t,t)_{a} (t,t)_{b} (t,t)_{c}}.$$
Define 
$$K_1(t) = \inf_n |(t,t)_n|, \quad K_2(t) = \sup_n |(t,t)_n|.$$
Since $(t,t)_n$ converge to $(t,t)_\infty$ uniformly on compact sets, the functions $K_1, K_2\colon \matD \to \matR$ are continuous. Since $(t,t)_\infty\neq 0$ and $(t,t)_n\neq 0$ for all $n$ we get 
$$0<K_1(t)\leqslant |(t,t)_n| \leqslant K_2(t).$$
Hence one immediately gets the claimed inequalities with $T_1(t)=K_1(t)^{-\frac{3}{2}}K_2(t)^2$ and $T_2(t)=K_2(t)^{-\frac{3}{2}}K_1(t)^2$.
It remains to estimate $\tetra$ from above, so we examine more closely the complicated formula for $\tetra$ given in Section \ref{ribbon:subsection}. Let $S_z$ be the summand corresponding to the value $z$ of the summation range; we get 
\begin{align*}
S_z & = \frac{t^{P(a,b,c,d,e,f,z)} \prod_{i=1}^3\prod_{j=1}^4 (t,t)_{\Box_i-\triangle_j}}{(t,t)_a(t,t)_b(t,t)_c(t,t)_d(t,t)_e(t,t)_f} \times \\
& \quad \frac{(-1)^z(t,t)_{z+1}}{
(t,t)_{z-\triangle_1} (t,t)_{ z-\triangle_2} (t,t)_{z-\triangle_3} (t,t)_{ z-\triangle_4} (t,t)_{\Box_1-z} (t,t)_{ \Box_2-z} (t,t)_{ \Box_3-z} (1-t) }
\end{align*}
where 
$$P(a,b,c,d,e,f,z) =\frac 32 z^2-\left((a+b+c+d+e+f)+\frac 12\right) z+ $$
$$ \frac{1}{8}\Big((a+b+c+d+e+f)^2+ab+ac+ae+af+bc+bd+bf+cd+ce+df+ef\Big).$$
Arguing as above we get:
$$|S_z|<|t|^{P(a,b,c,d,e,f,z)}\frac{K_2(t)^{13}}{K_1(t)^{14}}.$$
Now we claim that the function
$$z\mapsto P(a,b,c,d,e,f,z)$$ 
defined in the interval $\big[\max\{\triangle_i\},\min\{ \Box_j\}\big]$ attains its minimum at $z=\min\{\Box_j\}$. Indeed the derivative is strictly negative:
$$\frac{\partial P(a,b,c,d,e,f,z)}{\partial z}=3z-(a+b+c+d+e+f)-\frac 12<0$$
because $z\leqslant \Box_i$ for all $i$ implies that 
$$3z \leqslant \Box_1 + \Box_2 + \Box_3 = a+b+c+d+e+f.$$
Summing up, we get: 
$$\left|\sum_{z=\max\{\triangle_i\} }^{\min\{\Box_j\}}S_z\right|<\sum_z |t|^{P(a,b,c,d,e,f,z)}\frac{K_2(t)^{13}}{K_1(t)^{14}}\leqslant |t|^{P(a,b,c,d,e,f,\min\{\Box_j\})}\frac{K_2(t)^{13}}{K_1(t)^{14}}\cdot \frac 1{1-|t|}$$
where the last inequality follows from $\sum_n |t|^n = \frac 1{1-|t|}$ since the function $z\mapsto P(a,b,c,d,e,f,z)$ is strictly increasing.
We plug all the inequalities in Equation \eqref{eq:unit6j} and get:
$$\left|\begin{Bmatrix}  a & b & c \\ d & e & f \end{Bmatrix}^U\right| \leqslant |t|^{P(a,b,c,d,e,f,\min\{\Box_j\})+\frac{a+b+c+d+e+f}{4}} K(t).$$
Now we note that $\min\{\Box_j\} = \frac{C_2+C_3}2$ and hence the exponent
$$P(a,b,c,d,e,f,\min\{\Box_j\}) + \frac {a+b+c+d+e+f}4$$
equals
$$ \frac 38(C_2+C_3)^2 - \frac 12\left(C_1 + C_2 +C_3 + \frac 12\right)(C_2 + C_3) + $$
$$\frac 18\big((C_1+C_2+C_3)^2+C_1C_2 + C_2C_3+C_3C_1 + 2(C_1+C_2+C_3)\big) =$$
$$\frac 18 (C_1-C_2)(C_1-C_3)+\frac 14 C_1.$$
Therefore 
\begin{align*}
\left|\begin{Bmatrix}  a & b & c \\ d & e & f \end{Bmatrix}^U\right| & \leqslant |t|^{\frac 18 (C_1-C_2)(C_1-C_3)+\frac 14 C_1} K(t) \\
& = |A|^{\frac 12(C_1-C_2)(C_1-C_3) + C_1} K(A^4)
\end{align*}
\end{proof}

It is important to note that the function $K$ in Lemma \ref{FK:lemma} does not depend on the values $a,b,c,d,e,f$ of the unitary $6j$-symbol. We can use the lemma to prove a couple of estimates on the renormalized $6j$-symbol.

\begin{cor} \label{renormalized:cor}
There is a continuous function $K\colon \matD \to \matR_+$ such that the following inequality holds:
$$\left| \begin{Bmatrix}  a & b & c \\ d & e & f \end{Bmatrix}^R\right| \leqslant |A|^{\frac 12(C_1-C_2)(C_1-C_3)+C_1-c-f}K(A).$$
\end{cor}
\begin{proof}
It suffices to multiply the unitary $6j$-symbol by $\sqrt{\cerchio_c\cerchio_f}$ and get a contribution of type
$$\left|\sqrt{\cerchio_c\cerchio_f}\right| \leqslant |A|^{-c-f}K'(A)$$
for some continuous function $K'\colon \matD \to \matR_+$.
\end{proof}

\begin{cor}\label{analyticat0:cor}
The renormalized $6j$-symbol expands at $A=0$ as follows:
$$\begin{Bmatrix}  a & b & c \\ d & e & f \end{Bmatrix}^R 
= \left\{ \begin{array}{lr}
1 + O(A) & {\rm \ if\ } c+f = \max\{a+d, b+e\}, \\
O(A) & {\rm if\ } c+f \neq \max\{a+d, b+e\}.
\end{array} \right.
$$
\end{cor}
\begin{proof}
In the estimate of Corollary \ref{renormalized:cor} the exponent of $|A|$ is always non-negative:
$${\frac 12(C_1-C_2)(C_1-C_3)+C_1-c-f}\geqslant 0$$
since $C_1 = \max\{a+d, b+e, c+f\}\geqslant c+f$. The exponent vanishes if and only if 
$$c+f = C_1 = C_2,$$ 
that is precisely when $c+f = \max\{a+d, b+e\}$. The proof of Lemma \ref{FK:lemma} shows that the renormalized quantum $6j$-symbol is a finite sum (along the summand $z$) of rational functions, whose leading term at $A=0$ is uniquely determined by the dominating summand $z=\min\Box_j$. It therefore suffices to compute the coefficient of the leading term of this summand as
$$\frac{i^{c+f}}{i^{a+b+c+d+e+f}}  (-1)^{\min \Box_j} A^{\frac 12{(C_1-C_2)(C_1-C_3)+C_1}}.$$
When $c+f = \max\{a+d, b+e\}$ we get $\min \Box_j = \frac{a+b+d+e}2$ and hence the coefficient of the leading term is
$$\frac{i^{c+f}}{i^{a+b+c+d+e+f}}  (-1)^{\frac{a+b+d+e}2} =1.$$
Therefore the renormalized $6j$-symbol expands as $1+O(A)$ as stated.
\end{proof}
In particular the renormalized $6j$-symbol can be extended at $A=0$ to an analytic function on $\overline \matD \setminus F$. The following corollary concludes the proof of Proposition \ref{exist:renormalized:prop}.
\begin{cor} \label{Q:cor}
We have
$$\left| \begin{Bmatrix}  a & b & c \\ d & e & f \end{Bmatrix}^R\right| \leqslant |A|^{Q(a,b,c,d,e,f)}K(A)$$
for some non-negative function $Q\colon \matN^6 \to \matN$ such that for every $a,b,c,d,e,n\in \matN$ it holds 
$$\#\big\{f\ \big|\ Q(a,b,c,d,e,f)=n\big\}\leqslant 12.$$
\end{cor}
\begin{proof}
By Corollary \ref{renormalized:cor} it is sufficient to prove that the function 
$$Q(a,b,c,d,e,f):= \frac 12(C_1-C_2)(C_1-C_3)+C_1-c-f$$ 
is as claimed. 
Recall that $(C_1,C_2, C_3)$ are defined as the three-uple $(a+d,b+e,c+f)$ re-ordered in descending order. So there are $6$ cases for this re-ordering and in each such case, after fixing $a,b,c,d,e$ the function $f\mapsto Q(a,b,c,d,e,f)$ is polynomial of degree $\leqslant 2$ and non constant, hence it has at most $2$ preimages for every $n\in \mathbb{N}$.
\end{proof}

\section{Finite representations at roots of unity} \label{finite:section}
We construct here the analytic family $\{\rho_A\}_{A\in S}$ of finite-dimensional representations by readapting the same cocycle construction from Section \ref{repres:section} to a root-of-unity finite-dimensional context. We also prove the second statement of Theorem \ref{main2:teo}.

These finite-dimensional representations are not new: they are isomorphic to the ``Hom'' version of the well-known quantum representations constructed in \cite{BHMV}; in Theorem \ref{tqft:teo} we clarify the exact relation.

\subsection{Construction of the representations $\calH_r$}\label{hr:sub}
Let $A\in S$ be a fixed root of unity and set $r=r(A)$. Fix a triangulation $x$ of $\Sigma$. Recall from the introduction that $\calM_r^x$ is the set of all multicurves on $\Sigma$ having a representative that intersects every triangle of $x$ in at most $r-2$ arcs, and that $\calH_r^x\subset\calH$ is the subspace of all functions supported on $\calM_r^x$. (We include here the triangulation $x$ in the notation.)

Recall from Proposition \ref{normal:prop} that by isotoping a multicurve in normal position with respect to $x$ we get a 1-1 correspondence between multicurves and admissible colorings on $x$. Recall from Section \ref{reduced:subsection} that an admissible coloring is \emph{$r$-admissible} if 
$$a+b+c\leqslant 2(r-2)$$
at every triangle colored with $a,b,c$.
This condition implies in particular that $a\leqslant r-2$ for every colored edge $a$ on $x$. It is easy to check that the elements in $\calM_r^x$ correspond precisely to the $r$-admissible colorings of $x$. In particular, they are finite in number.

We now define for every pair $x,y$ of triangulations an isomorphism
$$c_A(y,x)\colon \calH_r^x \to \calH_r^y.$$
The isomorphism is defined exactly as in Section \ref{cocycle:subsection}: we first consider a pair of triangulations $x$, $y$ related by a flip as in Fig.~\ref{flip_colored:fig} and define
$$c_A(y,x)(\delta_\gamma) = \lambda_1\delta_{\delta_{\gamma_1}} + \ldots + \lambda_k\delta_{\gamma_k}$$
on every multicurve $\gamma$ corresponding to an $r$-admissible coloring as in Section \ref{cocycle:subsection}, the only difference being that we only consider multicurves $\gamma_i$ corresponding to $r$-admissible colorings of $y$. We define the coefficient $\lambda_i$ as above as the renormalized $6j$-symbol
$$\begin{Bmatrix} a & b& c \\ d & e & f \end{Bmatrix}^R = \frac{ \raisebox{-0.5cm}{\tetracolored }\sqrt{\cerchio_c\cerchio_f}}{\sqrt{\teta_{a,b,c}\teta_{a,e,f}\teta_{d,b,f}\teta_{d,e,c}}}$$
evaluated in $A$: this is well-defined because everything is $r$-admissibly colored and hence $A\not \in F$, see Definition \ref{renorm:defn}.

\begin{prop} The renormalized $6j$-symbols satisfy the orthogonality relation and the Biedenharn-Elliot identity.
\end{prop}
\begin{proof}
The renormalized $6j$-symbols satisfy these identities since the quantum $6j$-symbols do, by Proposition \ref{identities:red:prop}.
\end{proof}

We can now define as in Section \ref{repres:section} an isomorphism 
$$c_A(y,x):\calH_r^x \to \calH_r^y$$ 
by choosing a finite sequence $x=x_1,\ldots ,x_n=y$ of flips relating $x$ and $y$ and setting 
$$c_A(y,x)=c_A(x_n,x_{n-1})\circ \cdots \circ c_A(x_2,x_1).$$ 

\begin{prop} The linear map $c_A(y,x)$ is well-defined.
\end{prop}
\begin{proof}
The proof is analogous to that of Lemma \ref{bounded:cor}.
\end{proof}

The family of linear maps $c_A$ that we have just constructed is a cocycle in some weaker sense than that stated in Definition \ref{cocycle:defn}: points (1)-(3) are fulfilled, but each map $c_A(y,x)$ is only defined on some finite-dimensional subspace $\calH_r^x$ of $\calH$ which depends on $x$. We can analogously fix a triangulation $x$ and construct a finite-dimensional representation 
$$\rho_A\colon \MCG(\Sigma)\to \GL(\calH_r^x)$$ 
by setting 
$$\rho_A(g)=c_A(x,g x) \rho(g).$$

\subsection{Unitarity}\label{sub:unitarity}
The finite-dimensional representations just introduced are unitary only in some specific cases.

\begin{lemma}
If $A= \pm \exp(\frac{i\pi}{2r})$  
the renormalized $6j$-symbol
$$\begin{Bmatrix} a & b & c \\ d & e & f \end{Bmatrix}(A)$$
is a real number.
\end{lemma}
\begin{proof}
For such values of $A$ every quantum integer $[n]$ is a positive real number; the proof then goes exactly as when $A\in\matR$ in Lemma \ref{sign:lem}.
\end{proof}

\begin{cor}
If $A= \pm \exp(\frac{i\pi}{2r})$ the representation $\rho_A$ is unitary.
\end{cor}
\begin{proof}
The proof is analogous to that of Lemma \ref{orthogonal:lemma}.
\end{proof}

\subsection{Analytic convergence}\label{sub:teo1.5}
We can prove Theorem \ref{main2:teo} as a consequence of a lemma.

\begin{lemma} \label{analytic:lemma}
Let $x,y$ be two triangulations of $\Sigma$ and $v, v'$ be two vectors of $\check\calH$. There is a finite set $F\subset S$ such that the matrix coefficient
$$\langle c_A(x,y)v, v' \rangle$$
makes sense and varies analytically on $\overline \matD\setminus F$.
\end{lemma}
\begin{proof}
Since both $v$ and $v'$ are finite linear combinations of $\delta_\gamma$'s it suffices to consider the case $v=\delta_\gamma$ and $v'=\delta_{\gamma'}$ for some multicurves $\gamma$ and $\gamma'$. It also suffices to consider the case where $x$ and $y$ are related by a flip, since the general case easily follows.

The subspaces $\calH_r$ exhaust $\check\calH$ and hence both $\delta_{\gamma}$ and $\delta_{\gamma'}$ are contained in $\calH_r$ for sufficiently big $r$, hence the matrix coefficient makes sense after excluding finitely many values $F\subset S$ of $A$. When it exists, the matrix coefficient is either zero or a renormalized quantum $6j$-symbol evaluated at $A$, and in both cases it depends analytically on $A\in \overline \matD\setminus F$.
\end{proof}

\begin{rem} \label{poles:rem}
Lemma \ref{analytic:lemma} says that the matrix coefficient $f(A) = \langle c_A(x,y)v, v' \rangle$ 
coincides with an analytic function $h$ evaluated at $A$ except at finitely many points $F$. We note that various non-continuous behaviours can occur at a point $A_0\in F$, indeed all these possibilities can hold:
\begin{itemize}
\item the matrix coefficient $f(A_0)$ is not defined because $v$ does not belong to $\calH_{r(A)}$,
\item the matrix coefficient $f(A_0)$ is defined but $f(A_0) \neq h(A_0)$: in that case $h(A_0) = \lim_{A \to A_0} f(A)$ may be infinite or finite.
\end{itemize}
\end{rem}

When $v$ and $v'$ belong to the standard basis of $\check\calH$ the analytic coefficient is particularly nice:
\begin{lemma} \label{roots:lemma}
If $v = \delta_\gamma$ and $v' = \delta_{\gamma'}$ then $h(A) = \langle c_A(x,y)v, v' \rangle$ is the square root of a rational function, except on finitely many points $F\subset S$.
\end{lemma}
\begin{proof}
Let $\sigma$ and $\sigma'$ be the colorings on $Y$ corresponding to $\gamma$ and $\gamma'$. It is easy to prove by induction on the number of flips relating $x$ and $y$ that 
$$h(A) = g(A) \cdot \frac {\prod_{e}\sqrt{\cerchio_{e}}}{\prod_{v}\sqrt{\teta_{v}}} 
\cdot \frac {\prod_{e'}\sqrt{\cerchio_{e'}}}{\prod_{v'}\sqrt{\teta_{v'}}}.$$
where $g(A)$ is a rational function and the products are taken on the edges $e,v$ of $y$ colored by $\sigma$ and the edges $e',v'$ of $x$ colored by $\sigma'$. 
The ribbon graphs $\cerchio_e$ and $\teta_v$ are colored respectively as $e$ and as the edges incident to $v$.
\end{proof}

We can now deduce the second statement of Theorem \ref{main2:teo} from Lemma \ref{analytic:lemma}:

\begin{cor}\label{cor:main2teo}
For any pair $v,v'\in \check\calH$ of vectors and any element $g\in\MCG(\Sigma)$ there is a finite set $F \subset S$ such that the matrix coefficient 
$$\langle\rho_A(g)v,v'\rangle$$
is defined and varies analytically on $\overline \matD \setminus F$. 
\end{cor}
We can also prove half of Theorem \ref{faithfullness:teo}. Thanks to analyticity, faithfulness of the multicurve representation $\rho_0$ is enough to guarantee asymptotic faithfulness of the finite representations $\{\rho_A\}_{A \in S}$:

\begin{cor}
For every non-central element $g\in \MCG(\Sigma)$ there is a finite set $F\subset S$ such that $\rho_A(g)\neq \id$ for every $A\in S\setminus F$.
\end{cor}
\begin{proof}
We know that $\rho_0$ is the multicurve representation: since $g$ is non central there is a multicurve $\gamma$ such that $g(\gamma)\neq \gamma$ and hence 
$$\langle \rho_0(\delta_\gamma), \delta_\gamma \rangle = \langle \delta_{g(\gamma)}, \delta_\gamma \rangle = 0.$$ 

We know from Lemma \ref{roots:lemma} that after excluding finitely many values $F\subset S$ the function $h(A) = \langle \rho_A(g)(\delta_\gamma),\delta_\gamma\rangle$ is the square root of a rational function. Such a function is either constant or finite-to-one, hence $h^{-1}(1)$ contains either everything or finitely many points. The first case is excluded since $h(0)=0$ and therefore $\langle \rho_A(g)(v),v\rangle = 1$ only for finitely many values of $A$. Since $\langle \rho_A(g)(v),v\rangle = 1 \Leftrightarrow \rho_A(g)(v) = v$ we are done.
\end{proof}

\section{Re-interpretation in terms of the skein algebra of a surface}\label{relationskein:sec}
We have defined our family $\{\rho_A\}_{A\in \overline\matD \setminus S}$ of infinite-dimensional representations by fixing a Hilbert space $\calH=\ell^2(\calM)$ and perturbing the permutation representation $\rho_0$ by means of an analytic family $c_A$ of cocycles. The same family of representations on the dense subspace $\check \calH$ may be described from a different viewpoint which makes use of a well-known geometric object in quantum topology, the \emph{Kauffman skein algebra} of a surface. We introduce this viewpoint here, which will allow us to 
deduce easily that the representations are faithful (modulo the center of $\MCG(\Sigma)$) for all $A\in\matD$, a fact that is not obvious in the original context. 

\subsection{The skein algebra}
Let $A\in\matC^*$ be a non-zero complex number. We have defined in Section \ref{module:subsection} the Kauffman skein module $K_A(M)$ of $M$ as the $\matC$-vector space generated by all isotopy classes of framed links in $M$ modulo the Kauffman bracket relations shown in Fig.~\ref{Kauffman_bracket:fig}. 

Let $\Sigma$ be our punctured surface: in this section we really consider the marked points as punctures, so that $\Sigma$ is homeomorphic to the interior of a compact surface $\overline\Sigma$ with boundary. The skein module $K_A(\Sigma)$ is by definition 
$$K_A(\Sigma) = K_A(\Sigma \times [0,1]).$$
This module is equipped with a natural associative algebra structure over $\matC$: the product $L\cdot L'$ of two framed links $L$ and $L'$ is defined by taking $L\cup L'$ after pushing $L$ and $L'$ respectively inside $\Sigma\times [1-\varepsilon, 1]$ and $\Sigma\times [0, \varepsilon]$. This algebra is commutative if and only if $A=\pm 1$.

A multicurve in $\Sigma$ determines a framed link in $\Sigma\times \frac 12$ with the horizontal framing induced by $\Sigma$, and hence an element of $K_A(\Sigma)$. Przytycki has shown \cite{P} that multicurves generate $K_A(\Sigma)$ as a vector space:
\begin{prop}\label{prop:multicurvebasis}
Multicurves form a basis for the vector space $K_A(S)$.
\end{prop}

The mapping class group $\MCG(\Sigma)$ acts naturally on $K_A(S)$ as algebra morphisms. When $A$ is not a root of unity this nice structure is also enriched by a \emph{trace}, introduced by Frohman and Kania-Bartoszy\'nska in \cite{YM}.

\subsection{The Yang-Mills trace} \label{YM:subsection}
Every inclusion of manifolds $M\subset N$ induces a linear map $K_A(M)\to K_A(N)$. The inclusion which is of interest for us here is the following: the punctured $\Sigma$ is homeomorphic to the interior of a compact surface $\overline {\Sigma}$ and hence $\Sigma\times [0,1]$ is contained in the compact handlebody $\overline\Sigma \times [0,1]$ which is in turn contained in its double, homeomorphic to the connected sum $\#_k(S^2\times S^1)$ of some $k$ copies of $S^2\times S^1$, with $k=1-\chi(\Sigma)$. Summing up we get a linear map
$$K_A(\Sigma) \to K_A\big(\#_k(S^2\times S^1)\big)$$
induced by the inclusion $\Sigma \times [0,1]\subset \#_k(S^2\times S^1)$.
Let us suppose henceforth that $A$ is not a root of unity. With that hypothesis we may use another result of Przytycki \cite{HoPr, P2}:
\begin{prop}
Suppose $A\in \matC^*$ is not a root of unity. The space $K_A\big(\#_k(S^2\times S^1)\big)$ is one-dimensional and spanned by the class of the empty link.
\end{prop}

We have therefore constructed a linear map called the \emph{Yang-Mills trace} in \cite{YM}:
$$\YM\colon K_A(\Sigma) \to \matC$$
This map is indeed a trace, in the sense that the following equality holds
$$\YM(\alpha \cdot \beta) = \YM(\beta \cdot \alpha)$$
for any pair of skeins $\alpha$ and $\beta$. The Yang-Mills trace can then be used to define a complex bilinear form by setting
$$\langle \alpha, \beta \rangle = \YM(\alpha\cdot\beta).$$
Both the trace and the bilinear form depend on $A$ and we may use the symbols $\YM_A$ and $\langle,\rangle_A$ to stress this dependence. The mapping class group $\MCG(\Sigma)$ preserves the trace and hence the bilinear form.

\subsection{An orthogonal basis}
The multicurves basis for $K_A(\Sigma)$ is not orthogonal, \emph{i.e}~the quantity $\langle \gamma, \gamma'\rangle$ is often non-zero for a pair of multicurves $\gamma$ and $\gamma'$. We now describe an orthogonal basis for $K_A(\Sigma)$, which depends on the choice of a triangulation for $\Sigma$.

Let $x$ be a triangulation for $\Sigma$. By duality $x$ determines a trivalent spine $Y$ of $\Sigma$, which can be interpreted as a ribbon graph in $\Sigma \times \frac 12$ with the horizontal framing. We denote by $Y_\sigma$ the ribbon graph $Y$ equipped with an admissible coloring $\sigma$. Recall from Section \ref{ribbon:subsection} that  $S$ is the set of all roots of unity except $\pm 1$ and $\pm i$ and a colored ribbon graph determines a skein in $K_A(\Sigma)$ if $A\not \in S$.
 
\begin{prop} \label{basis:prop}
Suppose $A\in\matC^*\setminus S$. The elements $\{Y_\sigma\}$ where $\sigma$ varies among all admissible colorings of $Y$ form a basis for $K_A(\Sigma)$.
\end{prop}
\begin{proof}
This fact is well-known to experts since $Y$ is a spine of the handlebody $\overline\Sigma \times [0,1]$ and we may use \cite{L}; we give a simple proof for completeness and to prepare the reader to the similar Proposition \ref{finite:basis:prop} below.

Recall that every admissible coloring $\sigma$ induces an admissible coloring on the dual triangulation $x$ and hence determines a multicurve in normal form, and that this construction gives a bijection between admissible colorings and multicurves by Proposition \ref{normal:prop}. We put a partial ordering on the admissible colorings by saying that $\sigma\leqslant \sigma'$ if at every edge of $Y$ the coloring of $\sigma$ is smaller or equal than the coloring of $\sigma'$. This also induces a partial ordering on multicurves (which depends of course on the triangulation $x$).

Recall from Section \ref{sub:jw} that a Jones-Wentzl projector is of the form $1_n + i$ where $i$ is an element of the Temperly-Lieb algebra generated by the elements $e_j$ containing short returns. This fact easily implies that the skein $Y_\sigma$ equals the multicurve corresponding to the coloring $\sigma$ plus a linear combination of multicurves having strictly smaller colorings. Conversely, every multicurve equals the correspondingly colored ribbon graph $Y_\sigma$ plus a combination of colorings $Y_{\sigma'}$ with $\sigma'<\sigma$. We have just constructed a change of basis between multicurves and $\{Y_\sigma\}$.
\end{proof}

The transformation from the multicurve basis (which is canonical) to the basis $\{Y_\sigma\}$ (which depends on a triangulation $x$) can be seen as an orthogonalization, in virtue of the following.
 
\begin{prop} \label{orthogonal:prop}
Suppose $A\in\matC^*$ is not a root of unity.
Let $\sigma$ and $\sigma'$ be two admissible colorings for $Y$. The following equality holds:
$$\langle Y_\sigma, Y_{\sigma'} \rangle = \delta_{\sigma, \sigma'}\prod_{e}\cerchio_{e}^{-1}\prod_{v}\teta_{v} $$
The product is taken over all vertices $v$ and edges $e$ of $Y$ and the ribbon graphs $\cerchio_e$ and $\teta_v$ are colored respectively as $e$ and as the edges incident to $v$.
\end{prop}
\begin{proof}
The graph $Y$ is a spine of $\Sigma$ and is hence also a spine of the handlebody $\overline \Sigma \times [0,1]$: hence every edge of $Y$ intersects transversely in its midpoint a compressing disc of $\overline \Sigma \times [0,1]$, which doubles to a 2-sphere in the manifold $\#_k(S^2\times S^1)$.

\begin{figure}
 \begin{center}
  \includegraphics[width = 12.5 cm]{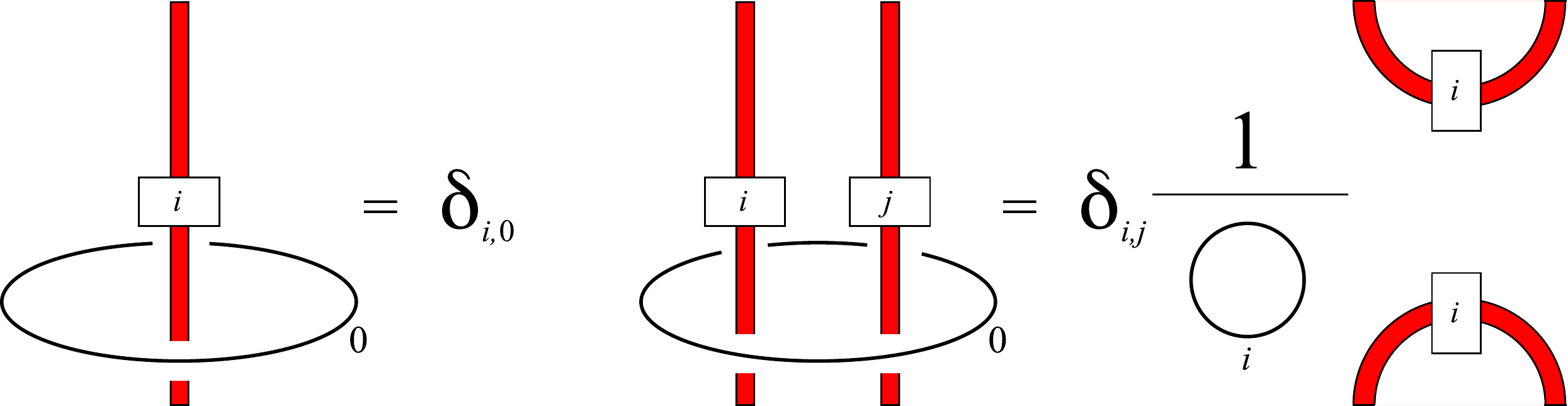}
 \end{center}
 \caption{Orthogonality of ribbon graphs intersecting an essential sphere.}
 \label{sphere:fig}
\end{figure}

We need to compute $\langle Y_\sigma, Y_{\sigma'} \rangle = \YM (Y_\sigma \cdot Y_{\sigma'}).$
Place two parallel copies of $Y_\sigma$ in $S\times [-1,1]$ one above the other and consider them as skeins in $K_A(\#_k(S^2\times S^1))$. Two parallel edges intersect in their midpoints a 2-sphere, and we can apply the move in Fig.~\ref{sphere:fig}-(right). Such a move is obtained by first performing a fusion as in Fig.~\ref{fusion:fig} and then using the following well-known fact \cite[Lemma 1]{YM}: any ribbon graph intersecting a 2-sphere in a single point with a non-zero coloring is equivalent to the empty skein $\emptyset$ (see Fig.~\ref{sphere:fig}-(left)).

If $\sigma\neq \sigma'$ the skein $Y_\sigma\cdot Y_{\sigma'}$ is hence equivalent to the empty skein. If $\sigma \neq \sigma'$ it is equivalent to a union of one $\teta_v$ graph sitting around each vertex $v$ of $Y$, divided by the evaluation of one $\cerchio_e$ at each edge $e$ of $Y$, thus yielding the equality of skeins:
$$Y_\sigma \cdot Y_\sigma = \left(\prod_{e}\cerchio_{e}^{-1}\prod_{v}\teta_{v} \right) \cdot \emptyset. $$
Since $\YM(k\cdot \emptyset) = k$ we are done.
\end{proof}

\begin{cor}
Suppose $A\in \matC^*$ is not a root of unity. The $\matC$-bilinear  
form $\langle, \rangle$ is non-degenerate.
\end{cor}
\begin{proof}
The evaluations $\cerchio_e$ and $\teta_v$ are product of quantum integers and their inverses and therefore their zeroes and poles are contained in $S\cup \{0\}$. Since $A\not \in S\cup \{0\}$ we have constructed an orthogonal basis $\{Y_\sigma\}$ with $\langle Y_\sigma, Y_\sigma \rangle \neq 0$, and this guarantees that $\langle, \rangle$ is non-degenerate.
\end{proof}

\begin{cor} \label{YM:cor}
We have $\YM(\emptyset) = 1$ and $\YM(Y_\sigma)=0$ for any non-trivial $\sigma$.
\end{cor}
\begin{proof}
Use $\YM(Y_\sigma) = \langle Y_\sigma, \emptyset \rangle$.
\end{proof}

In the definition of $\YM_A$ we needed to impose that $A$ is not a root of unity. We can now use Corollary \ref{YM:cor} as a definition to extend $\YM_A$ at the values $A=\pm 1$ and $\pm i$. (The move in Fig.~\ref{whitehead2:fig} easily shows that this definition does not depend on the triangulation $x$ as quantum integers are non-zero at $\pm1, \pm i$.) 

\subsection{Renormalized $6j$-symbols revisited} \label{revisited:subsection}
It is now easy to connect the representation of $\MCG(\Sigma)$ on $K_A(\Sigma)$ with the representation $\rho_A$ on the dense subspace $\check\calH$ defined in Section \ref{repres:section}. Suppose that $A\in \overline \matD^*\setminus  S$. Fix a triangulation $x$ and consider its dual ribbon graph $x$. We renormalize the skein $Y_\sigma$ by defining
$$\hat Y_\sigma = \frac {\prod_{e}\sqrt{\cerchio_{e}}}{\prod_{v}\sqrt{\teta_{v}}} Y_\sigma.$$
We use here the analytic functions $\sqrt{\cerchio}$ and $\sqrt {\teta}$ on $\overline \matD \setminus S$ introduced in Section \ref{renormalized:subsection}. Proposition \ref{orthogonal:prop} implies that
$$\langle \hat Y_\sigma, \hat Y_{\sigma'}\rangle = \delta_{\sigma, \sigma'}.$$
In other words $\{\hat Y_\sigma\}_\sigma$ form an orthonormal basis with respect to the complex bilinear form $\langle, \rangle$. The basis depend on a triangulation $x$, but the change of basis is a familiar formula:

\begin{lemma} \label{change:lemma}
Let $x$ and $x'$ be two triangulations related by a flip as in Fig~\ref{flip_colored:fig} and $\{\hat Y_\sigma\}, \{\hat Y'_{\sigma'}\}$ be the corresponding orthonormal bases. We have
$$\hat Y'_{\sigma'} = \sum_\sigma \begin{Bmatrix} a & b & c \\ d & e & f \end{Bmatrix}^R \hat Y_\sigma.$$
\end{lemma}
\begin{proof}
The very definition of the quantum $6j$-symbols from Fig.~\ref{whitehead2:fig} says that
$$Y'_{\sigma'} = \sum_\sigma \begin{Bmatrix} a & b & c \\ d & e & f \end{Bmatrix} Y_\sigma.$$
By renormalizing we replace each $Y$ with a $\hat Y$ and each quantum $6j$-symbol with a renormalized $6j$-symbol.
\end{proof}

Fix $A\in \overline\matD^*\setminus S$ and let 
$$\Psi_A\colon K_A(\Sigma) \to \check\calH$$
be the linear isomorphism which sends the skein $\hat Y_\sigma$ to $\delta_\gamma$, where $\gamma$ is the multicurve corresponding to the coloring $\sigma$.

\begin{cor} \label{psi:cor}
The map $\Psi_A$ defines an isomorphism between the representation $\rho_A^K$ and $\rho_A$ for any $A\in \overline \matD^* \setminus S$. 
\end{cor}

We therefore get the remaining half of Theorem \ref{faithfullness:teo}:

\begin{cor}\label{fidelity:cor}
Let $A\in \overline \matD\setminus S$. The representation $\rho_A$ on $\check \calH$ is faithful modulo the center.
\end{cor}
\begin{proof}
It is well-known that every non-central element $g\in \MCG(\Sigma)$ permutes some multicurve $\gamma$ and hence, by Proposition \ref{prop:multicurvebasis}, acts non-trivially on $K_A(\Sigma)$.
\end{proof}

\subsection{Unitary representations}\label{subsec:unitaryreps}
The bilinear product $\langle, \rangle$ on $K_A(\Sigma)$ does not induce a hermitian form on $K_A(\Sigma)$ for general values of $A$, but it does when $A\in \matR^*$, as we now shortly see.

Suppose $A\in\matR^*$. The skein relations in Fig.~\ref{Kauffman_bracket:fig} involve only real coefficients and may be used to define a real algebra $K_A^{\matR}(\Sigma)$ and a canonical isomorphism $K_A(\Sigma) = K_A^{\matR}(\Sigma)\otimes_\matR \matC$. The bilinear form $\langle \cdot, \cdot \rangle$ induces a real scalar product on $K_A^\matR(\Sigma)$ which extends to a \emph{hermitian} form $\langle \cdot, \cdot \rangle_H$ on the $\mathbb{C}$-vector space $K_A(\Sigma)$.

\begin{prop}
Suppose $A\in \matR^*$. The hermitian form $\langle, \rangle_H$ on $K_A(\Sigma)$ is positive-definite and $\{\hat Y_\sigma \}$ is an orthonormal basis.
\end{prop}
\begin{proof}
The element $Y_\sigma$ is real, \emph{i.e.}~it belongs to $K_A^{\matR}(\Sigma)$ because every quantum integer $[n]$ is a real number when evaluated at $A\in\matR^*$ and hence $Y_\sigma$ is a combination of framed links with real coefficients. 

We know from Proposition \ref{orthogonal:prop} that
$$\langle Y_\sigma, Y_{\sigma'} \rangle = \delta_{\sigma, \sigma'}\prod_{e}\cerchio_{e}^{-1}\prod_{v}\teta_{v}. $$
The sign of $\teta_{a,b,c}$ is $(-1)^{\frac{a+b+c}{2}}$ and the sign of $\cerchio_a$ is $(-1)^a$ so that by multiplying all these signs we get a contribution of $(-1)^{2a}=1$ for every edge of $Y$ colored with $a$, and hence $1$ as a result. Therefore $\langle, \rangle_H$ is positive-definite and indeed $\hat Y_\sigma$ is an orthonormal basis for $\langle, \rangle_H$.
\end{proof}

\begin{prop}\label{prop:isometryskein}
Suppose $A\in\matR^*$. The map $\Psi_A$ defines an isometry between $K_A(\Sigma)$ and $\check\calH$.
\end{prop}
\begin{proof}
The map sends the orthonormal basis $\{\hat Y_\sigma\}$ to the orthonormal basis $\{\delta_\gamma\}$.
\end{proof}

This concludes the proof of Proposition \ref{skein:prop}:

\begin{cor}
If $A\in [-1,0)\cup (0,1]$ the representation $\rho_A$ is isometric to the natural action of $\MCG(\Sigma)$ on the Hilbert space $\overline{K_A(\Sigma)}$.
\end{cor}

We now turn to Corollary \ref{character:cor}:

\begin{cor}
The representation $\rho_{-1}$ is isometric to the $\SU(2)$-character variety representation.
\end{cor}

\begin{proof}
Recall that 
$$X=\Hom\big(\pi_1(\Sigma),\SU(2)\big)/_{\SU(2)}$$
and the $\SU(2)$-character variety representation is the natural representation of $\MCG(\Sigma)$ on the Hilbert space $L^2(X,\mu)$ where $\mu$ is the probability measure induced by the Haar measure on $\SU(2)$. Bullock \cite{Bu} and Charles-March\'e \cite{ChMa} have constructed a $\MCG(\Sigma)$-equivariant algebra isomorphism 
$$i\colon K_{-1}(\Sigma) \longrightarrow T$$
where $T\subset L^2(X,\mu)$ is the algebra generated by the trace functions on $X$. The isomorphism $i$ sends a skein represented by a simple curve $\gamma$ to $-tr_{\gamma}(\cdot)$. 
The subspace $T$ is dense in $L^2(X,\mu)$: indeed $T$ coincides with the algebra of all regular functions on the $\SL_2(\matC)$-representation variety restricted to $X$, see \cite[Theorem 6.1]{Ma}, which is dense in $L^2(X,\mu)$ by Peter-Weyl theorem.

It only remains to verify that $i$ is an isometry. To do so, we note that both algebras $K_{-1}(\Sigma)$ and $T$ may be obtained by $\matC$-tensoring a real algebra $K_{-1}^\matR(\Sigma)$ and $T^\matR$: the former was noted at the beginning of Section \ref{subsec:unitaryreps} and the latter is due to the fact that trace functions are real on $X$. The scalar product on each algebra is then induced by a trace: it is the Yang-Mills trace for $K_{-1}^\matR(\Sigma)$ and the volume form $\int_X$ for $T^\matR$.
Therefore to prove that $i$ is an isometry it suffices to prove that $i$ preserves the traces, that is:
\begin{equation} \label{traces:eqn}
\YM(\alpha) = \int_X i(\alpha)
\end{equation}
for every skein $\alpha$. Let $Y$ be a trivalent spine for $\Sigma$. Since the ribbon graphs $\{Y_\sigma\}$ form a basis for $K_{-1}(\Sigma)$ and traces are linear, it suffices to prove (\ref{traces:eqn}) for $\alpha = Y_\sigma$. According to Corollary \ref{YM:cor} we need to prove that
$$\int_Xi(Y_0) = 1, \quad \int_Xi(Y_\sigma) = 0 \ \forall \sigma\neq 0$$
where $0$ denotes the zero coloring. The first equality is obvious since $i(Y_0) = i(\emptyset) = 1$ is the constant function and $\int_X 1 = 1$.

We now turn to the second equality. 
Set $G=\SU(2)$. Note that $X=G^{E(Y)}/G^{V(Y)}$ where $E(Y)$ and $V(Y)$ denote the set of edges and vertices in $Y$ respectively.
We use a formula provided in \cite{CoMa} (see formula (2), page 13 and the comments about inserting holonomies at the end of page 15) for the evaluation of the function $i(Y_{\sigma})$ at a point $g\in X$ which is represented by matrices $g_e,\ e\in E(Y)$: $$i(Y_\sigma)( \{g_e\}_{e\in E(Y)})=h\big(\bigotimes_{e\in E(Y)} (id_{\sigma(e)}\otimes \rho_{\sigma(e)}(g_e)) \omega_{\sigma(e)},\bigotimes_{v\in V(Y)} \epsilon_{\sigma(e_i),\sigma(e_j),\sigma(e_k)}\big)$$ where $h$ is a hermitian form on $\bigotimes_{e\in E(Y)} (V_{\sigma(e)}\otimes V_{\sigma(e)}^*)$ and $\omega_{\sigma(e)}$, $\epsilon_{\sigma(e_i),\sigma(e_j),\sigma(e_k)}$ are specific vectors belonging respectively to $V_{\sigma(e)}\otimes V_{\sigma(e)}^*$ and to $V^{\pm}_{\sigma(e_i)}\otimes V^{\pm}_{\sigma(e_j)}\otimes V^{\pm}_{\sigma(e_k)}$ (where $e_i,e_j,e_k$ are the edges touching $v$ and we use the notation $V^+_{\sigma(e_i)}=V_{\sigma(e_i)}$, $V^-_{\sigma(e_i)}=V^*_{\sigma(e_i)}$,  and the sign is $+$ iff $e_i$ ends at $v$). 
To write the above formula we implicitly lifted the function $i(Y_\sigma)$ to a $G^{V(Y)}$-invariant function on $G^{E(Y)}$. To conclude remark that by linearity one has:
\begin{align*}
\int_{G^{E(Y)}}h\big(\bigotimes_{e\in E(Y)} (id_{\sigma(e)}\otimes \rho_{\sigma(e)}(g_e)) \omega_{\sigma(e)},\bigotimes_{v\in V(Y)} \epsilon_{\sigma(e_i),\sigma(e_j),\sigma(e_k)}\big)=\\ 
h\big(\bigotimes_{e\in E(Y)} \big(id_{\sigma(e)}\otimes \int_{G_e} \rho_{\sigma(e)}(g_e)dg_e\big) \omega_{\sigma(e)},\epsilon_{\sigma(e_i),\sigma(e_j),\sigma(e_k)}
\big)\end{align*}
But the endomorphism $ \int_{G_e} \rho_{\sigma(e)}(g_e)dg_e$ of $V_{\sigma(e)}$ is $0$ as soon as $\sigma(e)\neq 0$. 
\end{proof}
\subsection{Finite-dimensional representations} \label{finite:subsection}
The finite-dimensional representations can be similarly interpreted using the \emph{reduced} Kauffman bracket.

In what follows we will always suppose that $A\in S$ is a primitive $(4r)^{\rm th}$ root of unity, that is $A=\exp(\frac{i\pi k}{2r})$, with $(k,2r)=1$. This mild hypothesis is sometimes needed to prove some crucial lemma and is hence usually assumed in the literature: see for instance Lemma 13.7 in \cite[Chapter 13]{L_book} or \cite{S}.

Roberts has proved that $K^{\rm red}_A(M)$ is a finite-dimensional vector space which depends only on $\partial M$ up to isomorphism (see \cite{S} for a proof), and that may be obtained only by killing the portions in Fig.~\ref{reduced:fig}-(left), the killing of right portions being redundant (unpublished).

In particular we have $K_A^{\rm red}(M) \isom K_A^{\rm red}(S^3) \isom \matC$ for every closed manifold $M$ including of course $\#_h(S^2\times S^1)$, and therefore we still have a Yang-Mills trace
$$\YM\colon K_A^{\rm red}(\Sigma) \longrightarrow K_A^{\rm red}(\#_h(S^2\times S^1)) \cong \matC$$
and hence a complex bilinear product $\langle,\rangle$ as in Section \ref{YM:subsection} \cite{YM}. Let $x$ be a triangulation for $\Sigma$ and $Y$ its dual ribbon graph. An analogue of Proposition \ref{orthogonal:prop} holds:
\begin{prop}\label{YMrootof1:prop}
Let $\sigma$ and $\sigma'$ be two $r$-admissible colorings for $Y$. The following equality holds:
$$\langle Y_\sigma, Y_{\sigma'} \rangle = \delta_{\sigma, \sigma'}\prod_{e}\cerchio_{e}^{-1}\prod_{v}\teta_{v} $$
\end{prop}
\begin{proof}
The proof of Proposition \ref{orthogonal:prop} applies also in this setting, since the moves in Fig.~\ref{sphere:fig} are also valid in the reduced skein module, see for instance \cite[Lemma 2]{S}.
\end{proof}

Multicurves form a canonical basis of $K_A(\Sigma)$. The reduced skein module $K_A^{\rm red}(\Sigma)$ does \emph{not} have a canonical basis: to construct some nice basis we need to fix a triangulation $x$ and its dual spine $Y$.

\begin{prop} \label{finite:basis:prop}
Each of the following finite sets form a basis for $K_A^{\rm red}(\Sigma)$:
\begin{enumerate}
\item the set of all multicurves that intersect (in normal form) every triangle of $x$ in $\leqslant r-2$ arcs, \label{multicurve:basis:item}
\item the set $\{Y_\sigma\}$ where $\sigma$ varies among all $r$-admissible colorings on $Y$.
\end{enumerate}
The second basis is orthogonal with respect to $\langle, \rangle$.
\end{prop}
\begin{proof}
Multicurves generate $K_A(\Sigma)$ and hence also generate $K_A^{\rm red}(\Sigma)$. Put every multicurve in normal form with respect to $x$ and order them as in Proposition \ref{basis:prop}: we say that $\gamma \leqslant \gamma'$ if their colors are $c\leqslant c'$ at every edge. If the multicurve intersects a triangle in more than $r-2$ arcs, we can use the killed portions in Fig.~\ref{reduced:fig} to express the multicurve as a linear combinations of smaller multicurves. Therefore every multicurve can be expressed as a linear combination of multicurves corresponding to $r$-admissible colorings on $x$, \emph{i.e.}~intersecting every triangle in $x$ in at most $r-2$ arcs. 

We claim that these multicurves are actually linearly independent, hence they are a basis. Indeed, arguing as in Proposition \ref{basis:prop}, one can express each multicurve $\gamma$ inducing an $r$-admissible coloring  $\sigma$ as $Y_\sigma$ plus a linear combination of colored ribbon graphs $Y_{\sigma'}$ with $\sigma'<\sigma$. But by Proposition \ref{YMrootof1:prop} the vectors $Y_\sigma$ are independent. 
\end{proof}

We can now define $\hat Y_\sigma$ renormalizing $Y_\sigma$ as in Section \ref{revisited:subsection}, so that Lemma \ref{change:lemma} also holds in this context, everything restricted only to $r$-admissible colorings. We can similarly define
$$\Psi_A\colon K_A^{\rm red}(\Sigma) \to \calH_r$$
by sending $\hat Y_\sigma$ to $\delta_\gamma$ where $\gamma$ is the multicurve corresponding to the coloring $\sigma$. Let $\rho_A^K$ denote the natural representation of $\MCG(\Sigma)$ on $K_A^{\rm red}(\Sigma)$.

\begin{cor} \label{isom:finito:cor}
The map $\Psi_A$ defines an isomorphism between the representation $\rho_A^K$ and $\rho_A$ for every $(4r)^{\rm th}$ root of unity $A$.
\end{cor}

\subsection{Finite dimensional representations and TQFT}
As mentioned in the introduction, the finite-dimensional representations $\{\rho_A\}_{A\in S}$ 
are not new, at least when $A$ is a primitive $(4r)$-th root of unity: they are determined by the well-known Reshetikhin-Turaev representations, see \cite{BHMV}. 
  
\begin{teo}\label{tqft:teo}
Let $A$ be a $(4r)^{\rm th}$ root of unity. The following isomorphism of representations of $\MCG(\Sigma)$ holds:
$$\calH_r\cong \bigoplus_{i_1,\ldots, i_n} \End(V(\Sigma ; i_1,\ldots, i_n))$$ where $i_1,\ldots, i_n$ range over all the colorings of the $n$ punctures of $\Sigma$ with values in $\{0,1,\ldots, r-2\}$ and $V(\Sigma;i_1,\ldots, i_n)$ is the module associated to $(\Sigma;i_1,\ldots, i_n)$ by the Reshetikhin-Turaev TQFT at level $r$. 
\end{teo}
\begin{proof}
Fix a handlebody $H$ with $\partial H = \Sigma$ and for each coloring $i_1,\ldots, i_n$ consider the corresponding reduced skein module $K_{A,i_1,\ldots,i_n}^{\rm red}(H)$ as a model for $V(\Sigma;i_1,\ldots, i_n)$. By pushing skeins inside $H$ we get an algebra map
$$\phi_{i_1,\ldots,i_n}\colon K_A^{\rm red}(\Sigma) \to \End(K_{A,i_1,\ldots,i_n}^{\rm red}(H)).$$
These maps collect as
$$\phi\colon K_A^{\rm red}(\Sigma) \to \bigoplus_{i_1,\ldots, i_n} \End(K_{A,i_1,\ldots,i_n}^{\rm red}(H)).$$
The map $\phi$ is surjective: to prove this fact one only needs to adapt Roberts' argument \cite[Theorem 6.5]{Ro} for closed surfaces to the punctured case. We conclude by proving that
$$\dim(K_A^{\rm red}(\Sigma))=\sum_{i_1,\ldots, i_n} \dim \End(K_{A,i_1,\ldots,i_n}^{\rm red}(H)).$$ By Proposition \ref{finite:basis:prop} the integer $\dim(K_A^{\rm red}(\Sigma))$ is the number of $r$-admissible colorings of a spine $Y$ of $\Sigma$ (considered as a punctured surface). Let $Z$ be a spine of $H$ having its 1-valent vertices at the marked points in the boundary $\partial H$: a basis for $K_{A,i_1,\ldots,i_n}^{\rm red}(H)$ is given by all the $r$-admissibly colored $Z$'s extending the fixed colors $i_1,\ldots, i_n$ at the marked points. Therefore $ \dim \End(K_{A,i_1,\ldots,i_n}^{\rm red}(H))$ is the square of such number, in other words it is the number of all $r$-admissible colorings of the graph $\overline Z$ obtained by doubling $Z$ along the 1-valent vertices, that are colored as $i_1,\ldots, i_n$ at the (now interior) marked points. 

Summing up on colorings we get that $\sum_{i_1,\ldots i_n} \dim \End(K_{A,i_1,\ldots,i_n}^{\rm red}(H))$ equals the number of $r$-admissible colorings of $\overline Z$. Since $\overline Z$ and $Y$ are trivalent graphs with the same Euler characteristic, they have the same number of $r$-admissible colorings (given by the Verlinde formula), and we are done.
\end{proof}

\section{Detection of pseudo-anosov maps}\label{sec:anosov}
In this section we prove Theorem \ref{main:dilatation:teo}. We will tacitly assume Corollary \ref{isom:finito:cor} and work with the natural representation $\rho_A^K$ of $\MCG(\Sigma)$ on the reduced skein module $K_A^{\rm red}(\Sigma)$ for some $(4r)^{\rm th}$ root of unity $A$. We will consider the marked points in $\Sigma$ as punctures, as customary in Nielsen-Thurston theory.

A key tool will be the following general lemma.

\begin{lemma} \label{key:lemma}
Let $\varphi\in \MCG(\Sigma)$ and $\gamma$ be a simple closed curve such that $\varphi(\gamma)$ and $\gamma$ are not isotopic. Let $x$ be an ideal triangulation for $\Sigma$ such that each triangle of $x$ intersects $\gamma$ and $\varphi(\gamma)$ in at most $N$ arcs and $A$ be a primitive $(4r)^{\rm th}$ root of unity. If $r>\frac N2 +1$ then $\rho_A^K(\varphi)\neq \id$.
\end{lemma}
\begin{proof}
Since $\gamma$ and $\varphi(\gamma)$ are not isotopic they are both nontrivial and can be isotoped into normal form with respect to $\Delta$, via an isotopy which does not increase the intersections of $\gamma$ and $\varphi(\gamma)$ with the edges of $x$: hence the resulting normal curves $\gamma$ and $\varphi(\gamma)$ intersect every triangle of $x$ again in at most $N$ arcs. These curves are therefore $r$-admissible (recall the definitions in Subsection \ref{hr:sub}) for any $r$ such that $N<2r-2$, \emph{i.e.} such that $r> \frac N2 +1$.
Hence they form different elements of the multicurve basis of $K_A^{\rm red}(\Sigma)$ described in Proposition \ref{finite:basis:prop}-(\ref{multicurve:basis:item}) and therefore $\rho_A(\varphi) \neq \id$.
\end{proof}

When $\varphi$ is pseudo-Anosov we now show how to construct an appropriate $\gamma$ and $x$ using train tracks. 

\subsection{Train tracks}
Let $\Sigma$ be as usual a punctured surface of negative Euler characteristic.
By the Nielsen-Thurston classification an element $\varphi\in\MCG(\Sigma)$ is either of finite order, reducible, or pseudo-Anosov. A pseudo-Anosov mapping class determines a (projective class of) stable $\calL^+$ and unstable $\calL^-$ measured laminations together with a real number $\lambda>1$ called \emph{dilatation}, such that $\varphi(\calL^+) = \lambda\calL^+$ and $\varphi(\calL^-) = \lambda^{-1}\calL^-$. Laminations are nicely encoded using some combinatorial objects called \emph{train tracks}.

A train track $\tau\subset \Sigma$ is a smooth complex where every vertex is trivalent and modeled as a \emph{switch}. The complement $\Sigma\setminus \tau$ is a disjoint union of connected surfaces with piecewise-smooth boundary, and we require that the double of each such surface along its smooth boundary has negative Euler characteristic. The edges of $\tau$ are called \emph{branches}.

A (transverse) \emph{measure} on a train track $\tau$ is the assignment of a non-negative real \emph{weight} to each branch which satisfies the switch condition $a = b+c$ at every vertex. A measure on $\tau$ determines a measured lamination in $\Sigma$. If the weights are natural numbers the measured lamination is simply a multicurve. The space of all weights on $\tau$ is a cone denoted by $V_\tau$, which can be seen as a subcone of the (piecewise-linear) space of all measured laminations in $\Sigma$.

A train track $\tau'$ is \emph{carried} by a train track $\tau$ if $\tau'$ may be smoothly immersed into $\tau$. If this holds we use the symbol $\tau' < \tau$ and notice that the smooth immersion induces an embedding of cones $V_{\tau'}\subset V_\tau$. An important case occurs when $\varphi$ is a diffeomorphism of $\Sigma$ and $\varphi(\tau)$ is carried by $\tau$: in that case $\varphi$ acts on the cone $V_\tau$ and this action is nicely encoded by a square \emph{incidence matrix}, whose definition we now recall following \cite{Pen}.

At every branch of $\tau$ we fix an interior point and a fiber in its tie-neighborhood, called the \emph{central tie} over the branch. We fix a smooth map $h\colon \Sigma\to \Sigma$ which homotopes $\varphi(\tau)$ inside $\tau$ keeping the branches transverse to the ties. Let $b_1,\ldots, b_n$ be the branches of $\tau$. The incidence matrix is a $n\times n$ matrix $M$ whose element $M_{ij}$ is the number of intersections of $h\circ \varphi(b_i)$ with $b_j$. (The matrix thus depends on the choice of $h$.)

The matrix $M$ is a nice and concrete object which describes the action of $\varphi$ on the cone $V_\sigma$. The following result was proved by Papadopoulos and Penner \cite[Theorem 4.1]{PP}.

\begin{teo} \label{track:teo}
Let $\varphi$ be a pseudo-Anosov diffeomorphism of $\Sigma$. There is a train track $\tau$ which carries $\varphi(\tau)$ and the incidence matrix is Perron-Frobenius.
\end{teo}

Recall that a $n\times n$ matrix $M$ with non-negative entries is Perron-Frobenius if some iterate $M^k$ has only strictly positive entries. A Perron-Frobenius matrix has a unique positive eigenvector up to scaling which corresponds here to the unstable lamination $\calL^+$. The dilatation $\lambda>1$ is its eigenvalue, which is also the largest real eigenvalue of $M$.
We will need the following result by Ham and Song \cite[Lemma 3.1]{HS}:

\begin{lemma} \label{PF:lemma}
Let $M$ be a $n\times n$ Perron-Frobenius matrix with integer entries with $\lambda > 1$ its largest eigenvalue. Then
$$\lambda^n \geqslant |M|-n+1$$
where $|M|$ denotes the sum of all entries of $M$.
\end{lemma}

\subsection{Finite representations}
Let $\Sigma$ be a punctured surface and $\varphi\in\MCG(\Sigma)$ be a pseudo-Anosov mapping class. If we puncture $\Sigma$ at a point inside each polygonal complementary region of $\calL^+$ (or equivalently $\calL^-$) we obtain a punctured surface $\Sigma^\circ$ and a restriction map $\varphi^\circ\colon \Sigma^\circ \to \Sigma^\circ$ well-defined up to isotopy, \emph{i.e.}~a new mapping class $\varphi^\circ \in \MCG(\Sigma^\circ)$. The map $\varphi^\circ$ is still pseudo-Anosov and has the same invariant laminations and dilatation $\lambda$ as the old map $\varphi$.
We first study the action of $\varphi^\circ$ on $K_A^{\rm red}(\Sigma^\circ)$:

\begin{prop} \label{dilatation:prop}
Let $A=\exp(\frac{i\pi k}{2r})$, with $(k,2r)=1$. If
$$r > \frac 12 \big(\lambda^{-3\chi(\Sigma^\circ)}-3\chi(\Sigma^\circ)+1\big)$$
 then $\rho_A^K(\varphi^\circ)\neq \id.$
\end{prop}
\begin{proof}
Theorem \ref{track:teo} says that there is a train track $\tau\subset \Sigma$ which carries $\varphi(\tau)$ with Perron-Frobenius $n\times n$ incidence matrix $M$, where $n$ is the number of edges of $\tau$. The complementary regions of $\calL^+$ are homeomorphic to the complementary regions of $\tau$, therefore $\tau$ is a spine of $\Sigma^\circ$. In particular the number $n$ of edges of $\tau$ may be calculated as $n=-3\chi(\Sigma^\circ)$.
The dilation $\lambda$ is the largest real eigenvalue of $M$ and Lemma \ref{PF:lemma} implies that 
$$\lambda^{-3\chi(\Sigma^\circ)} \geqslant |M|+3\chi(\Sigma^\circ)+1.$$
It is easy to prove that every train track $\tau$ carries a simple closed curve $\gamma$ which induces a weight 0, 1, or 2 on every edge of $\tau$ (simply start from any point of $\tau$ and keep walking until you come back to a point that you have already crossed in the same direction).
Let us identify $\gamma$ with its weights $n$-vector, having entries in $\{0, 1, 2\}$. The image $\varphi(\gamma) = M\gamma$ is another vector (because $M$ is Perron-Frobenius) and we have
$$\big|\varphi(\gamma)\big| = \big|M\gamma \big| \leqslant 2 |M| \leqslant 2\big(\lambda^{-3\chi(\Sigma^\circ)}-3\chi(\Sigma^\circ)-1\big)$$
where $|v|$ denotes again the sum of the entries of $v$. 

Let now $x$ be the ideal triangulation of $\Sigma^\circ$ dual to the 3-valent spine $\tau$. By construction $\gamma$ is a normal curve in $x$ which intersects the edges in $\leqslant 2$ points, and the switch condition $a=b+c$ holds at every triangle of $x$: therefore $\gamma$ intersects every triangle of $x$ in at most two arcs, and is hence $r$-admissible for any $r\geqslant 4$.

The curve $\varphi(\gamma)$ is also a normal curve in $x$ whose intersection numbers with the edges of $x$ sum up to a quantity smaller or equal than $K = 2\big(\lambda^{-3\chi(\Sigma^\circ)}-3\chi(\Sigma^\circ)-1\big)$.
The switch condition easily implies that the color at each edge is smaller or equal than $\frac K 2$, and that every triangle in $x$ intersects $\varphi(\gamma)$ in at most $\frac K 2$ arcs. By Lemma \ref{key:lemma} the endomorphism $\rho_A^K(\varphi)$ is non-trivial for all
\begin{align*}
r>\frac K4 +1 & = \frac 12 \big(\lambda^{-3\chi(\Sigma^\circ)}-3\chi(\Sigma^\circ)-1\big) +1 \\
& = \frac 12 \big(\lambda^{-3\chi(\Sigma^\circ)}-3\chi(\Sigma^\circ)+1\big).
\end{align*}
\end{proof}
A slightly weaker version of Proposition \ref{dilatation:prop} holds for the original pseudo-Anosov $\varphi$. 

\begin{teo} \label{dilatation:teo}
Let $\Sigma$ be a punctured surface and $\varphi\colon \Sigma \to \Sigma$ a pseudo-Anosov map with dilatation $\lambda > 1$. Let $A=\exp(\frac{\pi i k}{2r})$ with $(k,2r)=1$. If
$$r>-6\chi(\Sigma)\big(\lambda^{-9\chi(\Sigma)}-9\chi(\Sigma)-1\big) +1$$
then $\rho_A^K(\varphi)\neq \id$ 
\end{teo}
\begin{proof}
The proof of Theorem \ref{dilatation:prop} shows that there is an ideal triangulation $x^\circ$ for $\Sigma^\circ$ and a simple closed curve $\gamma \subset \Sigma^\circ$ such that both $\gamma$ and $\varphi^\circ(\gamma)$ intersect all the edges of $x^\circ$ in less than $K^{\circ}= 2\big(\lambda^{-3\chi(\Sigma^\circ)}-3\chi(\Sigma^\circ)-1\big)$ points. Since $\gamma$ is carried by a train track $\tau$ dual to $x_0$, it is a non-trivial curve in $\Sigma$.

The ideal triangulation $x^\circ$ of $\Sigma^\circ$ may not be an ideal triangulation of $\Sigma$ because it may contain $h> 0$ interior vertices which form precisely the set $\Sigma\setminus \Sigma^\circ$. The complementary regions of $\tau$ in $\Sigma$ correspond to the $h$ interior vertices in $\Sigma\setminus \Sigma^\circ$. The definition of train track forces each such complementary region to be a surface with piecewise smooth boundary, whose double has negative Euler characteristic: a simple Euler characteristic count then shows that $h \leqslant -2\chi(\Sigma)$ and hence $\chi(\Sigma^\circ)= \chi(\Sigma)-h \geqslant 3\chi(\Sigma)$. So $\gamma$ and $\varphi^\circ(\gamma)$ intersect the edges of $x^{\circ}$ in at most 
$$K=2\big(\lambda^{-9\chi(\Sigma)}-9\chi(\Sigma)-1\big) \geqslant  2\big(\lambda^{-3\chi(\Sigma^\circ)}-3\chi(\Sigma^\circ)-1\big) $$ 
points.

We now construct an ideal triangulation $x$ for $\Sigma$ starting from $x^\circ$ as follows. The triangulation $x^\circ$ has some $n$ ideal vertices $p_1,\ldots, p_n$ and some $h$ interior vertices. Let $T_1,\ldots T_n$ be some maximal disjoint subtrees of the $1$-skeleton of $x^\circ$ such that $p_i\in T_i$. Every vertex of $x^\circ$ is contained in a unique tree $T_i$. 
By collapsing each tree $T_i$ to its root $p_i$ we transform the triangulation $x^\circ$ into an ideal cellularization of $\Sigma$, possibly containing some bigons or monogons, which can be collapsed recursively to get a true ideal triangulation $x$ of $\Sigma$: each $2$-cell of the cellularization has at most $3$-sides and collapsing monogons or bigons creates cellularizations with the same properties; when no monogons and bigons are left, one gets a triangulation.

We claim that if a normal multicurve $\gamma$ intersects $t$ times $x^\circ$ then it will intersect $x$ in at most $-6\chi(\Sigma)t$ points: indeed before the collapse one can isotope $\gamma$ away from the trees and the intersections will at most get multiplied by twice the number of edges of $x$ (contractions of monogons and bigons only decrease the intersections), which is $-6\chi(\Sigma)$. 

We have found a non-trivial curve $\gamma$ such that both $\gamma$ and $\varphi(\gamma)$ intersect the edges of $x$ in at most 
$$-12\chi(\Sigma)\big(\lambda^{-9\chi(\Sigma)}-9\chi(\Sigma)-1\big)$$ 
points. Since $\gamma$ is non-trivial and $\varphi$ is pseudo-Anosov the curves $\varphi(\gamma)$ and $\gamma$ are not isotopic and we may apply Lemma \ref{key:lemma}, hence we get $\rho_A^K(\varphi)\neq\id$ as soon as
$$r > -6\chi(\Sigma)\big(\lambda^{-9\chi(\Sigma)}-9\chi(\Sigma)-1\big)+1.$$ 

\end{proof}

\section{Irreducible components at $A=0$} \label{sec:irreducibles}

We prove here Proposition \ref{orthogonalsubspaces:prop} and Theorem \ref{irreducibles:teo}. 

\subsection{$\matZ_2$-homology}
For any element $\alpha \in H_1(\Sigma, \matZ_2)$ define $\calM^\alpha\subset \calM$ as the set of all multicurves that represent $\alpha$. The splitting $\calM = \cup_{\alpha\in H_1(\Sigma, \matZ_2)} \calM^\alpha$ induces an orthogonal splitting
$$\calH = \bigoplus_{\alpha \in H_1(\Sigma, \matZ_2)} \calH^\alpha$$
where
$\calH^\alpha\subset \calH$ consists of all functions that are supported only on $\calM^\alpha$.
We define analogously $\calH_r^{x,\alpha} = \calH_r^x \cap \calH^\alpha$.

\begin{prop} Let $x$ and $y$ be ideal triangulations of $\Sigma$. The map $c_A(y,x)\colon \calH \to \calH$ preserves each factor $\calH^\alpha$ when $A\in \overline\matD \setminus S$ and sends $\calH_r^{x,\alpha}$ into $\calH_r^{y,\alpha}$ when $A\in S$ with $r=r(A)$.
\end{prop}
\begin{proof}
We can suppose without loss of generality that $x$ and $y$ are related by a flip as in Fig.~\ref{flip:fig}: the multicurves determined by the left and right colouring in Fig.~\ref{flip:fig} are homologous.
\end{proof}

We get Proposition \ref{orthogonalsubspaces:prop} as a corollary.

\begin{cor}
The closed subspaces $\calH^0$ and $\calH^{\neq 0}$ are $\rho_A$-invariant for every $A \in \overline \matD\setminus S$. Similarly $\calH_r^0$ and $\calH^{\neq 0}_r$ are $\rho_A$-invariant for any $A\in S$ with $r(A)=r$.
\end{cor}
\begin{proof}
The image $\rho_A(g)$ of a mapping class $g\in\MCG(\Sigma)$ is defined as $c_A(x,gx)\rho_0(g)$, and both $c_A(x,gx)$ and $\rho_0$ preserve the two subspaces.
\end{proof}

\subsection{The multicurve representation}
We now turn to the multicurve representation $\rho_0$. 
It splits into infinitely many representations
$$\calH = \oplus_O \calH_O$$
corresponding to the splitting of $\calM = \cup_{O\in \calO} \calM_O$ into orbits. 

\begin{teo}
Every factor $\calH_{O}$ splits into finitely many orthogonal irreducibles. 
\end{teo}
\begin{proof}
Let $\gamma\in \calM_O$ be a fixed multicurve and $\St_\gamma<\MCG(\Sigma)$ its stabilizer. 

Recall that the \emph{commensurator} $\Comm_G(H)$ of a subgroup $H<G$ is the subgroup of $G$ consisting of all the elements $g\in G$ such that $g^{-1}Hg\cap H$ has finite index in both $g^{-1}Hg$ and $H$. It is easy to prove that if $H'<H$ is a finite-index subgroup then
$$\Comm_G(H') = \Comm_G(H).$$
Of course $\Comm_G(H)>H$ and a standard theorem of Mackey \cite{Mac} states that the unitary action of $G$ on $\ell^2(G/H)$ is irreducible if and only if $\Comm_G(H) = H$: when this holds we say that $H$ is \emph{self-commensurating}. If $H$ has finite index in $\Comm_G(H)$ the representation of $G$ on $\ell^2(G/H)$ splits into finitely many orthogonal irreducibles, see \cite[Section 2]{BD}. 

Since $\calM_O$ is in natural 1-1 correspondence with $\MCG(\Sigma)/\St_\gamma$
the action of $\MCG(\Sigma)$ on $\overline{\calH_O}$ is isometric to the unitary action on $\ell^2(\MCG(\Sigma)/\St_\gamma)$. It remains to prove that $\St_\gamma$ has finite index in its commensurator.

Fr every multicurve $\gamma$ we define the multicurve $\gamma^\circ$ as the multicurve obtained from $\gamma$ by removing some components as follows:
\begin{itemize}
\item delete all components that bound a disc with one marked point;
\item take only one representative for any maximal set of parallel curves.
\end{itemize}
As a result, every component in $\Sigma\setminus \gamma^\circ$ has negative Euler characteristic and more importantly $\gamma^\circ$ identifies a simplex in the curve complex $C(\Sigma)$. Paris has shown \cite{Pa} using a result of Burger - De La Harpe \cite{BD} that the stabilizers of the action of $\MCG(\Sigma)$ on $C(\Sigma)$ are self-commensurating, that is we have
$$\Comm_{\MCG(\Sigma)}(\St_{\gamma^\circ}) = \St_{\gamma^\circ}.$$
The stabilizer $\St_\gamma$ is a finite-index subgroup of $\St_{\gamma^\circ}$ and this implies that they have the same commensurator $\St_{\gamma^\circ}$. Therefore $\St_\gamma$ is a finite-index subgroup in its commensurator and hence the unitary representation on $\ell^2(\MCG(\Sigma)/\St_\gamma)$ splits into finitely many orthogonal irreducible components.
\end{proof}

\end{document}